\documentclass[10pt]{amsart}
\usepackage{verbatim}
\usepackage{amssymb}
\usepackage{amsbsy}
\usepackage{amscd}
\usepackage{amsmath}
\usepackage{amsthm}
\usepackage{color}
\usepackage{version}

\newenvironment{NB}{
\color{red}{\bf NB}. \footnotesize 
}{}
\excludeversion{NB}
\newenvironment{NB2}{
\color{blue}{\bf NB}. \footnotesize
}{}

\theoremstyle{plain}
 \newtheorem{thm}{Theorem}[section]
 \newtheorem{lem}[thm]{Lemma}
 \newtheorem{prop}[thm]{Proposition}
 \newtheorem{cor}[thm]{Corollary}
 
\theoremstyle{definition}
 \newtheorem{defn}{Definition}[section]
\theoremstyle{remark}
 \newtheorem{rem}{Remark}[section]
 \newtheorem{ex}{Example}[section]
 
\def\Bbb{\mathbb}

\def\cal{\mathcal}

\newcommand{ \Supp}{\operatorname{Supp}}
\newcommand{\Ext}{\operatorname{Ext}}
\newcommand{\Hom}{\operatorname{Hom}}
\newcommand{\codim}{\operatorname{codim}}

\newcommand{\rk}{\operatorname{rk}}

\newcommand{\NS}{\operatorname{NS}}
\newcommand{\coker}{\operatorname{coker}}
\newcommand{\Pic}{\operatorname{Pic}}
\newcommand{\ch}{\operatorname{ch}}
\newcommand{\td}{\operatorname{td}}
\newcommand{\Alb}{\operatorname{Alb}}
\newcommand{\Hilb}{\operatorname{Hilb}}
\newcommand{\Quot}{\operatorname{Quot}}

\newcommand{\Coh}{\operatorname{Coh}}

\newcommand{\Div}{\operatorname{Div}}

\newcommand{\id}{\operatorname{id}}
\newcommand{\topo}{\operatorname{top}}
\newcommand{\NSf}{\operatorname{NS_{\mathrm{f}}}}

\font\b=cmr10 scaled \magstep5
\def\bigzerou{\smash{\lower1.7ex\hbox{\b 0}}}

\numberwithin{equation}{section}
\begin{document}

\title{Moduli spaces of stable sheaves on Enriques surfaces}
\author{K\={o}ta Yoshioka}
\address{Department of Mathematics, Faculty of Science,
Kobe University,
Kobe, 657, Japan
}
\email{yoshioka@math.kobe-u.ac.jp}

\thanks{
The author is supported by the Grant-in-aid for 
Scientific Research (No.\ 26287007,\ 24224001), JSPS}

\begin{abstract}
We shall study the existence condition of $\mu$-stable
sheaves on Enriques surfaces.
We also give a different proof of the irreducibility
of the moduli spaces of rank 2 stable sheaves.
\end{abstract}

\maketitle

\renewcommand{\thefootnote}{\fnsymbol{footnote}}

\section{Introduction}
Let $X$ be an Enriques surface defined over an algebraically 
closed field $k$ of characteristic $\ne$ 2.
Moduli spaces of stable sheaves on $X$ are studied by
various people. 
In particular, the non-emptiness of the moduli
spaces is completely determined 
(\cite{Kim4},\ \cite{N},\cite{Enriques}) and the irreducibility of the 
moduli spaces was proved if $X$ is unnodal and the associated
Mukai vector  is primitive (\cite{Kim4}, \cite{N}, \cite{Y:twist1}).
In this note, we shall discuss the existence problem of
$\mu$-stable sheaves
on Enriques surfaces.
We also give a remark on the irreducibility of the moduli spaces.

For a coherent sheaf $E$ on $X$ or an element $E$ of $K(X)$,
let $v(E):=\ch(E)\sqrt{\td_X} \in H^*(X,{\Bbb Q})$
be the Mukai vector of $E$.
We introduce the Mukai pairing on $H^*(X,{\Bbb Q})$ by
$\langle x,y \rangle:=-\int_X x^{\vee} \wedge y$,
where for $x=(x_0,x_1,x_2) \in H^*(X,{\Bbb Q})$, $x^{\vee}:=(x_0,-x_1,x_2)$.
Then  
\begin{equation}
(v(K(X)),\langle \;\;,\;\;\rangle) 
\end{equation}
is the Mukai lattice of $X$.
For a Mukai vector $v$,
we use the following two expressions
$$
v=(r,\xi,a)=r+\xi+a \varrho_X, \;
r \in {\Bbb Z}, \xi \in \NSf(X),\frac{r}{2}+a  \in {\Bbb Z},
$$
where $\varrho_X \in H^4(X,{\Bbb Z})$ is the fundamental class of $X$
and $\NSf(X)$ the torsion free quotient of $\NS(X)$, 
that is, $\NSf(X)=\NS(X)/{\Bbb Z}K_X$.

For the Mukai vector $v \in H^*(X,{\Bbb Q})$
of a torsion free sheaf,
we assume that a polarization $H$ is general with 
respect to $v$ (see Definition \ref{defn:mu-stability}).
We would like to remark that
the problem of constructing 
$\mu$-stable locally free sheaves
was studied by Kim in the rank 2 case, and by  
Nuer \cite{N} in the rank 4 case.

For a Mukai vector $v$,
${\cal M}(v)$ denotes the moduli stack of coherent sheaves $E$ with
$v(E)=v$.
Let $H$ be an  ample divisor on $X$.
${\cal M}_H(v)^{ss}$ (resp.  ${\cal M}_H(v)^{s}$)
denotes the substack of ${\cal M}(v)$
consisting of (Gieseker) semi-stable sheaves
(resp. stable sheaves). 
Let $\overline{M}_H(v)$ be the moduli scheme of $S$-equivalence
classes of semi-stable sheaves and
$M_H(v)$ the open subscheme consisting of stable sheaves.
If $v=(r,\xi,a)$ with $r>0$, then 
${\cal M}_H(v)^{\mu ss}$ (resp.  ${\cal M}_H(v)^{\mu s}$)
denotes the substack of ${\cal M}(v)$
consisting of $\mu$-semi-stable sheaves
(resp. $\mu$-stable sheaves). 
As in \cite{Y:twist1},
we also introduce ${\cal M}_H(v,L)^{ss}$ 
(resp. ${\cal M}_H(v,L)^{s}$, $\overline{M}_H(v,L), M_H(v,L)$) 
as the locus of ${\cal M}_H(v)^{ss}$
(resp. ${\cal M}_H(v)^{s}$, $\overline{M}_H(v), M_H(v)$)
consisting of $E$ with $c_1(E)=L$ in $\NS(X)$,
where $[L \mod K_X] =\xi$,

For a K3 surface, 
the existence condition of  $\mu$-stable sheaves
was completely described in \cite{Y:irred}.
For an Enriques surface, we get a similar result.
\begin{thm}[{Theorem \ref{thm:mu-s}}]\label{thm:intro:mu-s}
Let $v=(lr,l\xi,\frac{s}{2})$ be a Mukai
vector such that $\gcd(r,\xi)=1$ and $\langle v^2 \rangle \geq 0$.
Let $H$ be a general polarization with respect to $v$.
Then 
${\cal M}_H(v,L)^{ss}$ contains a $\mu$-stable
sheaf if and only if 
\begin{enumerate}
\item[(i)]
There is no stable sheaf $E$ such that 
$v(E)=(r,\xi,b)$ and 
$\langle v(E)^2 \rangle=-1,-2$,
and
$\langle v^2 \rangle \geq 0$ or
\item[(ii)]
There is a stable sheaf $E$ such that 
$v(E)=(r,\xi,b)$ and  
$\langle v(E)^2 \rangle=-1$, and
$\langle v^2 \rangle \geq l^2$ or
\item[(iii)]
There is a stable sheaf $E$ such that 
$v(E)=(r,\xi,b)$ and $\langle v(E)^2 \rangle=-2$, and
$\langle v^2 \rangle \geq 2l^2$.
\end{enumerate}
Moreover if $lr>1$, then
under the same condition,
${\cal M}_H(v,L)^{ss}$ contains a $\mu$-stable locally free
sheaf.
\end{thm}

In the second part, we shall study the irreducibility of the
moduli spaces ${\cal M}_H(v,L)^{ss}$.
The irreducibility of these moduli spaces on an arbitrary 
surfaces was proved by
Gieseker and Li \cite{G-L} and O'Grady \cite{O} when the expected dimension
$d$ is larger than a constant $N(r)$ that depends only on the rank $r$.
We can expect a better estimate for $N(r)$ in the Enriques case, as occurs for
K3 and abelian surfaces, 
since an Enriques surface also has a numerically trivial canonical divisor.
Let $v=(r,\xi,\frac{s}{2})$ be a primitive Mukai vector on an Enriques surface.
If $r$ is odd, then
the irreducibility of ${\cal M}_H(v,L)^{ss}$ was proved in \cite{Y:twist1}.  
If $r=2$, then
the irreducibility was investigated by
Kim \cite{Kim4} and Nuer \cite{N} if $X$ is unnodal.
Kim reduced the problem to the cases
where $s=1,2$ and proved the irreducibility
for $s=1$.
For the second case, Nuer 
reduced to the first case.
Then by using Bridgeland stability, Nuer \cite[Thm. 1.1]{N}
 completed the proof of irreducibility 
for the even rank case.
In this note, we shall give a different proof of the 
irreducibility for $r=2$. 
Combining a deformation argument with results in \cite{Y:twist1} 
and \cite{Enriques},
we get the following result.
\begin{thm}\label{thm:intro:irred}
Let $v=(r,\xi,\frac{s}{2})$ 
be a primitive Mukai vector on an Enriques surface $X$ and
$L$ a divisor on $X$ with $[L \mod K_X]=\xi$.
\begin{enumerate}
\item[(1)]
${\cal M}_H(v,L)^{ss}$ is connected for a general $H$.
\item[(2)]
If  $X$ is unnodal or $\langle v^2 \rangle \geq 4$, 
then ${\cal M}_H(v,L)^{ss}$ is irreducible for a general $H$.
\end{enumerate}
\end{thm}

The strategy of our proof is the same as our proof for the similar problem
on K3 surfaces \cite[Thm. 3.18]{Y:twist1}.
Thus we reduce the problem to the moduli of 
stable 1-dimensiional sheaves by a relative Fourier-Mukai transform
associated to an elliptic fibration.
Then by a detailed estimate of the locus 
of stable sheaves whose supports are 
reducible or nonreduced,
we show that the moduli space is birationally equivalent
to an abelian fiber space over an open subset of a projective space.    
Unlike the case of K3 surfaces, 
we need to require the Mukai vector to be primitive.  
Indeed if the Mukai vector is of the form
$v=m(r,\xi,\frac{s}{2})$ ($m,r \in {\Bbb Z}_{>0}$, $\xi \in \NS(X)$,
$2 \mid r-s$ and $2 \nmid r$),
then we can not reduce to the rank 0 case.
Moreover it is not so easy to study 
stable 1-dimensional sheaves on non-reduced curves.
Hence we can only treat 1-dimensional sheaves on non-reduced
curves of multiplicity 2,
which is sufficient to treat the primitive case.
We give partial generalizations in Remark \ref{rem:non-prim1}
and Remark \ref{rem:non-prim2}.
In the course of the proof, we also show that 
\cite[Assumption 2.16]{Sacca} holds for $v=(0,\xi,\frac{s}{2})$
such that $\xi$ is primitive (Proposition \ref{prop:sacca}). 
In particular $b_2(M_H(v,L))=11$
if $X$ satisfies \eqref{eq:rho=10} and $v=(r,\xi,\frac{s}{2})$ is
a primitive Mukai vector such that $2 \mid r$, $2 \nmid \xi$ and
$\langle v^2 \rangle \geq 4$ by \cite[Thm. 5.1]{Sacca} and
\cite{Enriques}.

\vspace{1pc}

{\it Acknowledgement.}
I would like to thank Howard Nuer for answering my question
on $\mu$-stable locally free sheaves of rank 2 and valuable and many 
comments
on the first version of this article.

\vspace{1pc}

{\it Notation.}

For an Enriques surface $X$, 
let $\varpi:\widetilde{X} \to X$ be the covering K3 surface
and $\iota:\widetilde{X} \to \widetilde{X}$
the covering involution.
For a primitive Mukai vector $v=(r,c_1,\frac{s}{2})$,
we set $\ell(v):=\gcd(r,c_1,s)$.
Then $\ell(v)=1,2$.

\subsection{Stabilities and their moduli stacks}\label{subsect:def}

Let $X$ be a smooth projective surface and $H$ an ample
divisor on $X$.

\begin{defn}\label{defn:mu-stability}
\begin{enumerate}
\item[(1)]
A torsion free sheaf $E$ is $\mu$-semi-stable (resp. $\mu$-stable)
if 
\begin{equation}
\frac{(c_1(F),H)}{\rk F} \underset{(<)}{\leq} \frac{(c_1(E),H)}{\rk E}
\end{equation}
for any subsheaf $F$ of $E$ with $0<\rk F <\rk E$. 
\item[(2)]
A polarization $H$ is general with respect to $v$,
if  
for any $\mu$-semi-stable sheaf $E$ with $v(E)=v$
and any subsheaf $F$ of $E$,
\begin{equation}
\frac{(c_1(F),H)}{\rk F}=\frac{(c_1(E),H)}{\rk E}
\Longleftrightarrow 
\frac{c_1(F)}{\rk F}=\frac{c_1(E)}{\rk E}.
\end{equation}
\end{enumerate}
\end{defn}

\begin{defn}
Let $G$ be an element of $K(X)$ with $\rk G>0$.
\begin{enumerate}
\item[(1)]
A torsion free sheaf $E$ is $G$-twisted semi-stable
(resp. $G$-twisted stable),
if
\begin{equation}
\frac{\chi(G,F(nH))}{\rk F} 
\underset{(<)}{\leq} \frac{\chi(G,E(nH))}{\rk E}
\;\; (n \gg 0)
\end{equation}
for any subsheaf $F$ of $E$ with $0<\rk F <\rk E$.
\item[(2)]
A purely 1-dimensional sheaf $E$ is 
$G$-twisted semi-stable
(resp. $G$-twisted stable),
if
\begin{equation}
\frac{\chi(G,F)}{(c_1(F),H)} \underset{(<)}{\leq} 
\frac{\chi(G,E)}{(c_1(E),H)}
\end{equation}
for any proper subsheaf $F \ne 0$ of $E$.
\item[(3)]
Since $G$-twisted semi-stability depends only on
$v(G)$,
we also define $w$-twisted semi-stability as 
a $G$-twisted semi-stability, where $v(G)=w$.  
\item[(4)]
${\cal M}_H^G(v)^{ss}$ (resp. ${\cal M}_H^G(v)^s$)
denotes the moduli stack of $G$-twisted semi-stable sheaves
(resp. $G$-twisted stable sheaves). 
\end{enumerate}
\end{defn}

\begin{rem}
\begin{enumerate}
\item[(1)]
$G$-twisted semi-stability depends only on 
$c_1(G)/\rk G$. 
\item[(2)]
If $H$ is general with respect to $v$,
then
$G$-twisted semi-stability is independent of the choice of
$G$.
\end{enumerate}
\end{rem}

\begin{defn}
For a coherent sheaf $E$ on $X$,
$\rk E$ denotes the rank of $E$.
For a purely 1-dimensional sheaf $E$ on $X$,
$\Div(E)$ denotes the effective divisor $C$ such
that $E$ is an ${\cal O}_C$-module and $c_1(E)=C$.
For a purely 1-dimensional sheaf $E$ on a smooth projective surface
$X$, $\Div(E)$ denotes the scheme-theoretic support of $E$.
Let 
$$
0 \to V_{-1} \overset{\varphi}{\to} V_0 \to E \to 0 
$$
be a locally free resolution of $E$. Then
the Cartier divisor $\det \varphi$ is $\Div(E)$.
\end{defn}

Let us recall a quotient stack description 
of ${\cal M}_H(v)^{\mu ss}$ and its open substacks.
For an ample divisor $H'$ on $X$,
let $Q(mH',v)$ be the open subscheme of 
the quot-scheme $\Quot_{{\cal O}_X(-mH')^{\oplus N}/X}$
consisting of points
\begin{equation}\label{eq:quotient}
\lambda:{\cal O}_X(-mH')^{\oplus N} \to E
\end{equation}
such that
\begin{enumerate}
\item
$v(E)=v$,
\item
$\lambda$ induces an isomorphism
$H^0(X,{\cal O}_X^{\oplus N}) \cong H^0(X,E(mH'))$,
\item
$H^i(X,E(mH'))=0$, $i>0$.
\end{enumerate}
Let ${\cal O}_{Q(mH',v) \times X}(-mH')^{\oplus N} \to {\cal Q}_v$
be the universal quotient.
We set $V_v:={\cal O}_X(-mH')^{\oplus N}$.
For our purpose, the choice of $mH'$ is not so important.
Hence we simply denote $Q(mH',v)$ by $Q(v)$.
Let ${\cal M}(v)$ be the moduli stack of coherent sheaves
$E$ with $v(E)=v$ and
$q_v:Q(v) \to {\cal M}(v)$ be the natural map.
We denote the pull-backs $q_v^{-1}({\cal M}_H(v)^{\mu ss}),
 q_v^{-1}({\cal M}_H(v)^{ss}),\ldots$ by
$Q(v)^{\mu ss}, Q(v)^{ss}, \dots $ respectively.
If we choose a suitable $Q(v)$, then
$q_v:Q(v)^{\mu ss} \to {\cal M}_H(v)^{\mu ss}$ is 
surjective and 
${\cal M}_H(v)^{\mu ss}$ is a quotient stack of $Q(v)^{\mu ss}$
by the natural action of $GL(N)$:
\begin{equation}
 {\cal M}_H(v)^{\mu ss} \cong \left[Q(v)^{\mu ss}/GL(N) \right].
\end{equation}
From now on, we assume that $q_v:Q(v)^{\mu ss} \to {\cal M}_H(v)^{\mu ss}$
is surjective.
We have 
\begin{equation}
\dim {\cal M}_H(v)^{\mu ss}=\dim Q(v)^{\mu ss}-\dim GL(N).
\end{equation}
\begin{rem}
Since $PGL(N)$ acts freely on $Q_H(v)^s$, we have
$\dim {\cal M}_H(v)^s=\dim M_H(v)-1$.
\end{rem}
\begin{lem}\label{lem:stack-dim}
Let ${\cal M}$ be an irreducible component of
${\cal M}_H(v)^{\mu ss}$. Then
$\dim {\cal M} \geq \langle v^2 \rangle$.
\end{lem}

\begin{proof}
We take a quotient \eqref{eq:quotient}.
Then we see that $\Ext^2(\ker \lambda,E)=0$.
By the deformation for the quot-scheme,
the Zariski tangent space of 
the quot-scheme at \eqref{eq:quotient} is $\Hom(\ker \lambda, E)$ and
the obstruction space is $\Ext^1(\ker \lambda,E) \cong \Ext^2(E,E)$.
Hence the dimension of an irreducible component
of $Q(v)^{\mu ss}$ containing the point
\eqref{eq:quotient}
is at least of 
$$
\dim \Hom(\ker \lambda,E)-\dim\Ext^1(\ker \lambda,E)
=N^2-\chi(E,E)=
\langle v^2 \rangle+\dim GL(N).
$$ 
Hence we get the claim.
\end{proof}

The following formula is used frequently in this paper.
\begin{lem}[{\cite[Lem. 5.2]{K-Y}}]\label{lem:A-est1}
Let ${\cal F}(v_1,v_2)$ be the stack of filtrations
$0 \subset E_1 \subset E$ such that
$E_1$ is a coherent sheaf with $v(E_1)=v_1$ and
$E_2:=E/E_1$ is a coherent sheaf with $v(E_2)=v_2$.
We have a morphism 
$p_{v_1,v_2}:{\cal F}(v_1,v_2) \to {\cal M}(v_1) \times {\cal M}(v_2)$
by sending  $E_1 \subset E$ to $(E_1,E/E_1)$.
We set
\begin{equation}
 \begin{split}
  {\cal N}^n(v_1,v_2):&=\{(E_1,E_2) \in {\cal M}(v_1) \times 
  {\cal M}(v_2)|
  \dim \Hom(E_1,E_2(K_X))=n \},\\
  {\cal F}^n(v_1,v_2):&=p_{v_1,v_2}^{-1}({\cal N}^n(v_1,v_2))\\
  &=\{(F_1 \subset E) \in {\cal F}(v_1,v_2)| \dim \Hom(F_1,(E/F_1)(K_X))=n \}.
 \end{split}
\end{equation}
Then
\begin{equation}
 \dim {\cal F}^n(v_1,v_2)=\dim {\cal N}^n(v_1,v_2)+
 \langle v_1,v_2 \rangle+n.
\end{equation} 
\end{lem}

\begin{proof}
Since $\dim \Ext^2(E_2,E_1)=\dim \Hom(E_1,E_2(K_X))=n$,
the same proof of \cite[Lem. 5.2]{K-Y} works.
\end{proof}

\begin{NB}
Let
$v=(lr,l\xi,\frac{s}{2})$ $(l,r,s \in {\Bbb Z}, \xi \in \NS(X))$ 
be a Mukai vector such that $\gcd(r,\xi)=1$.
Then $v$ is primitive if and only if
\begin{enumerate}
\item
$r$ is odd, $l-s$ is even and
$\gcd(l,\frac{l-s}{2})=1$ or
\begin{NB2}
We set $r=2k+1$, then
$\gcd(l,\frac{lr-s}{2})=\gcd(l,\frac{l-s}{2})$.
If $lr$ is odd, then we may 
state that $s$ is odd and $\gcd(l,s)=1$.
We set  $l=s+2k$. Then
$\gcd(l,s)=\gcd(2k,s)=\gcd(k,s)$, since $s$ is odd.
By $\gcd(l,\frac{l-s}{2})=\gcd(l,k)=\gcd(s,k)$,
$\gcd(l,s)=\gcd(l,\frac{l-s}{2})$.
\end{NB2}
\item
$r$ is even and
$\gcd(l,\frac{s}{2})=1$.
\end{enumerate}

\begin{thm}
Let $v=(lr,l\xi,\frac{s}{2})$ be a primitive Mukai
vector such that $\gcd(r,\xi)=1$ and $\langle v^2 \rangle \geq 0$.
For $L \in \NS(X)$ with $[L \mod K_X] =l \xi$,
${\cal M}_H(v,L)^{ss}$ contains a $\mu$-stable
sheaf if and only if 
\begin{enumerate}
\item[(i)] 
There is a stable sheaf $E$ such that 
$v(E)=(r,\xi,b)$ and  
$\langle v(E)^2 \rangle=-1$, and
$\langle v^2 \rangle \geq l^2$ or
\item[(ii)] 
There is a stable sheaf $E$ such that 
$v(E)=(r,\xi,b)$ and $\langle v(E)^2 \rangle=-2$, and
$\langle v^2 \rangle \geq 2l^2$ or
\item[(iii)] 
There is no stable sheaf $E$ such that 
$v(E)=(r,\xi,b)$ and 
$\langle v(E)^2 \rangle=-1,-2$
and
$\langle v^2 \rangle \geq 0$.
\end{enumerate}
Moreover if $lr>1$, then
under the same condition,
${\cal M}_H(v,L)^{ss}$ contains a $\mu$-stable locally free
sheaf.
\end{thm}
\end{NB}

\section{The dimension of moduli stacks}

In this section, we assume that $X$ is an Enriques surface,
and we shall estimate the dimension of various
substacks of ${\cal M}_H(v)$.
We also showed that
${\cal M}_H(v)^{ss}$ is a reduced stack if $\langle v^2 \rangle>0$
or $v$ is a primitive and isotropic Mukai vector.

\begin{lem}\label{lem:s}
Let $v=(r,\xi,\frac{s}{2})$ 
be a Mukai vector with $\langle v^2 \rangle>0$. Then
\begin{enumerate}
\item[(1)]
${\cal M}_H(v,L)^s$ is reduced and  
$\dim {\cal M}_H(v,L)^s=\langle v^2 \rangle$.
\item[(2)]
${\cal M}_H(v,L)^s$ is normal, unless
\begin{enumerate}
\item[(i)] $v=2v_0$ with $\langle v_0^2 \rangle=1$ and 
$L \equiv \frac{r}{2}K_X \mod 2$
or
\item[(ii)] $\langle v^2 \rangle=2$.
\end{enumerate}
\end{enumerate}
\end{lem}

\begin{proof}
We set 
\begin{equation}
{\cal M}_H(v,L)_{sing}^s:=\{E \in {\cal M}_H(v,L)^s \mid E \cong E(K_X) \}.
\end{equation}
We set $v=(r,c_1,\frac{s}{2})$.
If $r$ is odd, then ${\cal M}_H(v,L)_{sing}^s= \emptyset$.
Hence we assume that $r$ is even.
By \cite{Kim2} (see also Remark \ref{rem:Kim}) or \cite{Yamada},
$\dim {\cal M}_H(v,L)_{sing}^s$ is odd and
$\dim {\cal M}_H(v,L)_{sing}^s \leq \frac{\langle v^2 \rangle}{2}+1$.
Moreover if the equality holds, then
$2 \mid c_1$ and $L \equiv \frac{r}{2}K_X \mod 2$, and
if $v$ is primitive, $\langle v^2 \rangle \equiv 0 \mod 8$ 
(see Remark \ref{rem:rho=10}. 
In particular, we have 
$$
\langle v^2 \rangle-\dim {\cal M}_H(v,L)_{sing}^s
\geq 
\begin{cases}
\frac{\langle v^2 \rangle}{2}-1,& \langle v^2 \rangle \equiv 0 \mod 4,\\
\frac{\langle v^2 \rangle }{2},& \langle v^2 \rangle \equiv 2 \mod 4.
\end{cases}
$$
Since ${\cal M}_H(v,L)^s \setminus {\cal M}_H(v,L)_{sing}^s$ is smooth of
dimension $\langle v^2 \rangle$,
by the proof of Lemma \ref{lem:stack-dim},
we see that ${\cal M}_H(v,L)^s$ is a locally complete intersection stack of
$\dim {\cal M}_H(v,L)^s =\langle v^2 \rangle$.
In particular, ${\cal M}_H(v,L)^s$ is reduced.

(2) In order to prove the normality of
${\cal M}_H(v,L)^s$, it is sufficient to prove
$\langle v^2 \rangle-\dim {\cal M}_H(v,L)_{sing}^s \geq 2$.
If $\langle v^2 \rangle \geq 6$, then
obviously $\langle v^2 \rangle-\dim {\cal M}_H(v,L)_{sing}^s \geq 2$.
If $\langle v^2 \rangle=4$ and $v$ does not satisfy (i), then 
$\dim {\cal M}_H(v,L)_{sing}^s<
\frac{\langle v^2 \rangle}{2}+1=3$, which implies
$\dim {\cal M}_H(v,L)_{sing}^s \leq 1$.
In particular, $\langle v^2 \rangle-\dim {\cal M}_H(v,L)_{sing}^s \geq 3$.
Therefore (2) holds.
\begin{NB}
If $4 \mid \langle v^2 \rangle$ and
$\dim {\cal M}_H(v)_{sing}^s=\langle v^2 \rangle/2+1$, then
$\ell(v)=2$. In particular,
if $v$ isisotropic, then $\dim {\cal M}_H(v)^s=1$ implies
$\ell(v)=2$. 
\end{NB}
\begin{NB}
Assume that $v=2v_0$ with $\langle v_0^2 \rangle=1$.
\end{NB}
\begin{NB}
If $\ell(v)=1$, then ${\cal M}_H(v)^{ss}$ is smooth of dimension 
$\langle v^2 \rangle$ by \cite{N}.
If $\ell(v)=2$, then
we see that $8 \mid \langle v^2 \rangle$.
$\dim {\cal M}_H(v)^{ss}=\langle v^2 \rangle$.
If $v$ is not primitive, then
$\langle v^2 \rangle \geq 4$. By using Lemma \ref{lem:yamada} again,
we have $\dim {\cal M}_H(v)^s =\langle v^2 \rangle$.
\end{NB}
\end{proof}

\begin{rem}[{Nuer \cite{N}, Sacca \cite{Sacca}}]\label{rem:rho=10}
Assume that 
\begin{equation}\label{eq:rho=10}
\varpi^*(\Pic(X))=\Pic(\widetilde{X}),
\end{equation} 
thus
$\iota$ acts on $\Pic(\widetilde{X})$ trivially.

Let $v:=(r,\xi,\frac{s}{2})$ be a primitive Mukai vector.
Then ${\cal M}_H(v,L)^s$ is smooth of 
$\dim {\cal M}_H(v,L)^s=\langle v^2 \rangle$, unless
$\ell ((r,\xi,\frac{s}{2}))=2$ and 
$L \equiv \frac{r}{2}K_X \mod 2$.

\begin{proof}
By using $(-1)$-reflection (see Remark \ref{rem:Kim}),
we may assume that ${\cal M}_H(v,L)^s$ consists of
$\mu$-stable locally free sheaves.
Assume that $E \cong E(K_X)$.
Then $r$ is even and
there is a locally free sheaf $F$ such that
$\varpi_*(F)=E$. Then $\varpi^*(E)\cong F \oplus \iota^*(F)$.
By our assumption on $X$,
$\iota^*(c_1(F))=c_1(F)$ and $c_1(F)=c_1(\varpi^*(L))$, 
where $L \in \Pic(X)$.
Hence $\det F=\varpi^*(L)$ and
$\varpi^*(c_1(E))=c_1(F)+\iota^*(c_1(F))=2\varpi^*(c_1(L))$.
Then 
$$
c_1(E)=c_1(\varpi_*(F))=c_1(\varpi_*(\det(F)))
+\left(\frac{r}{2}-1 \right)K_X=
2c_1(L)+\frac{r}{2}K_X
$$
by \cite[Lem. 3.5]{Enriques}.
Hence $c_1(E) \equiv \frac{r}{2}K_X \mod 2$.
\end{proof}
\end{rem}

\begin{rem}\label{rem:Kim}
In Kim's paper \cite{Kim2},
it is assumed that $E \in {\cal M}_H(v)_{sing}^s$
is locally free.
We can reduce the general case to the case of 
$\mu$-stable locally free sheaves
by using $(-1)$-reflection 
(see \cite[Rem. 2.19]{Enriques}).
\end{rem}

\begin{NB}
If $r$ is odd, then ${\cal M}_H(v)^s$ is smooth.
If $r$ is even, then $\langle v^2 \rangle$ is even.
If $\langle v^2 \rangle=2$, then 
$\dim {\cal M}_H(v)_{sing}^s \leq \langle v^2 \rangle=2$.
\end{NB}

\begin{lem}\label{lem:pss}
Let $v$ be a Mukai vector with $\langle v^2 \rangle>0$.
We set
\begin{equation}
{\cal M}_H(v)^{pss}:=\{E \in {\cal M}_H(v)^{ss} \mid
\text{$E$ is properly semi-stable }\}.
\end{equation}
Assume that $H$ is general with respect to $v$.
Then
\begin{enumerate}
\item[(1)]
$\dim {\cal M}_H(v)^{pss} \leq \langle v^2 \rangle-1$.
Moreover $\dim {\cal M}_H(v)^{pss} \leq \langle v^2 \rangle-2$ unless
$v=2v_0$ with $\langle v_0^2 \rangle=1$.
 \item[(2)]
${\cal M}_H(v)^{s} \ne \emptyset$ and 
$\dim {\cal M}_H(v)^{ss}=\langle v^2 \rangle$.
\end{enumerate}
\end{lem}

\begin{proof}
We set $v=lv_0$ where $v_0$ is primitive and $l \in {\Bbb Z}_{>0}$.
We first note that the first claim of (1) 
implies (2) by Lemma \ref{lem:stack-dim},
Lemma \ref{lem:s} and \cite{N} (see also \cite[Thm. 3.1]{Enriques}).
The proof of (1) is almost the same as of \cite[Lem. 3.2]{K-Y}.
So I only remark that \cite[Lem. 5.1]{K-Y} is replaced by
Lemma \ref{lem:stack-dim} and
\cite[(3.4)]{K-Y} is replaced by
$$
\dim J(v_1,v_2) \leq \langle v^2 \rangle-
(\langle v_1,v_2 \rangle-\max\{l_2/l_1-1,0\}),
$$
where $J(v_1,v_2)$ is the substack whose member $E$ 
fits in an exact sequence
\begin{equation}
0 \to E_1 \to E \to E_2 \to 0
\end{equation}
such that $E_1$ is a stable sheaf with 
$v_1:=l_1 v_0$ and $E_2$ is a semi-stable sheaf with 
$v_2:=l_2 v_0$.  
\begin{NB}
For a general stable sheaf $E_1$,
$\Hom(E_1,E_2(K_X))=0$.
\end{NB}
We first assume that $\langle v_0^2 \rangle \geq 2$.
Then
$$
\langle v_1,v_2 \rangle-\max\{l_2/l_1-1,0\}=
l_1 l_2 \langle v_0^2 \rangle-\max\{l_2/l_1-1,0\}
\geq 2.
$$
Hence 
$\dim J(v_1,v_2) \leq \langle v^2 \rangle-2$, so the second claim of (1) holds if
$\langle v_0^2 \rangle>1$, and a fortiori the
first claim.
So we may assume that $\langle v_0^2 \rangle=1$. Then
the first claim of (1) clearly holds from the dimension estimate on
$J(v_1,v_2)$, so let us prove the second claim.
We set 
$$
H_k:=\{ (E_1,E_2) \mid E_1 \in {\cal M}_H(v_1)^s, 
E_2 \in {\cal M}_H(v_2)^{ss}, \dim \Hom(E_1,E_2 (K_X))=k \}.
$$
For $(E_1,E_2) \in H_k$, 
$E_2/(E_1(K_X)^{\oplus k})$
is a semi-stable sheaf
with the Mukai vector $v_2-k v_1$ by \cite[Lem. 3.1]{K-Y}.
Hence $E_1$ is determined by $E_2$ 
as a factor of a Jordan-H\"{o}lder filtration of $E_2$. 
Moreover if $k \geq 2$, then 
$E_2$ is properly semi-stable.
Therefore $\dim H_1 \leq \langle v_2^2 \rangle$
and $\dim H_k \leq \langle v_2^2 \rangle-1$ for $k>1$. 
If $k \geq 2$, then
\begin{equation}
\begin{split}
\langle v_1,v_2 \rangle+k+\dim H_k 
\leq & \langle v^2 \rangle-
(\langle v_1,v_2 \rangle+\langle v_1^2 \rangle+1-k)\\
\leq & \langle v^2 \rangle-2.
\end{split}
\end{equation}
If $k=1$ and $l_1 l_2 \geq 1$, then
\begin{equation}
\begin{split}
\langle v_1,v_2 \rangle+k+\dim H_k 
\leq & \langle v^2 \rangle-
(\langle v_1,v_2 \rangle+\langle v_1^2 \rangle-1)\\
\leq & \langle v^2 \rangle-2.
\end{split}
\end{equation}
Therefore 
$\dim J(v_1,v_2) \leq 
\max_k (\dim H_k+\langle v_1,v_2 \rangle+k)
\leq \langle v^2 \rangle-2$ if $l_1 l_2 \geq 2$.
The remaining case is $v=2v_0$, which is excluded.
Therefore (1) holds.
\end{proof}

\begin{cor}\label{cor:normal/reduced}
Let $v=(r,\xi,\frac{s}{2})$ 
be a Mukai vector with $\langle v^2 \rangle>0$. Then
for a general polarization $H$, we have the following.
\begin{enumerate}
\item[(1)]
${\cal M}_H(v,L)^{ss}$ is reduced and  
$\dim {\cal M}_H(v,L)^{ss}=\langle v^2 \rangle$.
\item[(2)]
${\cal M}_H(v,L)^{ss}$ is normal, unless
\begin{enumerate}
\item[(i)] $v=2v_0$ with $\langle v_0^2 \rangle=1$ and 
$L \equiv \frac{r}{2}K_X \mod 2$ or
\item[(ii)] $\langle v^2 \rangle=2$.
\end{enumerate}
\end{enumerate}
\end{cor}

\begin{proof}
By Lemma \ref{lem:s} and Lemma \ref{lem:pss}, (1) holds.
Moreover (2) also holds unless $v=2v_0$ with $\langle v_0^2 \rangle=1$ 
or (ii) $\langle v^2 \rangle=2$.
Therefore we shall treat the moduli stack
${\cal M}_H(2v_0,L)^{ss}$ with 
$\langle v_0^2 \rangle=1$ and $L \not \equiv \frac{r}{2}K_X \mod 2$.
By Lemma \ref{lem:s} (2),
 ${\cal M}_H(2v_0,L)^s$ is normal.

We shall prove that ${\cal M}_H(2v_0,L)^{ss}$ is smooth
in a neighborhood of the boundary.
Since $\langle v_0^2 \rangle =1$, $\rk v_0$ is odd, which implies that
$\frac{r}{2}K_X \equiv K_X \mod 2$.
Since $2 \mid \xi$ in $\NSf(X)$, we have
$L =2D, 2D+K_X$ $(D \in \NS(X))$.
Therefore $L=2D$ by
$L \not \equiv \frac{r}{2}K_X \equiv K_X \mod 2$.
Assume that $E \in {\cal M}_H(2v_0,L)^{ss}$ is $S$-equivalent to
$E_1 \oplus E_2$. By $\det E_1=(\det E_2^{\vee})(L)$ and $\rk v_0$ is odd,
$\Hom(E_i,E_j(K_X)) =0$ for all $1 \leq i,j \leq 2$.
Thus $\Ext^2(E,E)=\Hom(E,E(K_X))^{\vee}=0$, which implies that
${\cal M}_H(2v_0,L)^{ss}$ is smooth at $E$.
Therefore ${\cal M}_H(2v_0,L)^{ss}$ is a normal stack.
\begin{NB}
In particular, ${\cal M}_H(u,0)^{ss}$ is normal,
where $u=(2,0,-1)$.
\end{NB}
\end{proof}

\begin{NB}
We set 
${\cal M}_H(v)_{sing}^s:=\{E \in {\cal M}_H(v)^s \mid E \cong E(K_X) \}$.
\begin{lem}[{\cite[Lem. 2.13]{Yamada}}]\label{lem:yamada}
If $\langle v^2 \rangle \geq 4$, then
$\dim {\cal M}_H(v)_{sing}^s \leq \langle v^2 \rangle-1$.
Moreover if 
 $\langle v^2 \rangle \geq 6$, then
$\dim {\cal M}_H(v)_{sing}^s \leq \langle v^2 \rangle-2$.
\end{lem}
\end{NB}

\subsection{Isotropic case}
Let $v=(r,\xi,\frac{s}{2})$ be a primitive and isotropic Mukai vector
and take a general polarization $H$ with respect to $v$.

\begin{lem}\label{lem:isotropic}
\begin{enumerate}
\item[(1)]
If $\ell(v)=1$, then 
${\cal M}_H(v)^s$ is a reduced stack of 
$\dim {\cal M}_H(v)^s=\langle v^2 \rangle=0$.
\item[(2)]
If $\ell(v)=2$ and ${\cal M}_H(v)^s \ne \emptyset$, then
${\cal M}_H(v)^s$ is smooth of $\dim {\cal M}_H(v)^s=1$.
\end{enumerate}
\end{lem}

\begin{proof}
Assume that $\ell(v)=1$.
Since $2 \nmid \xi$,
by the proof of Lemma \ref{lem:s}, 
we see that
$\dim {\cal M}_H(v)_{sing}^s=-1$.
Thus ${\cal M}_H(v)^s$ is reduced and
$\dim {\cal M}_H(v)^s=0$.

If $\ell(v)=2$ and ${\cal M}_H(v)^s \ne \emptyset$, then
there is a 2 dimensional component of the
moduli space $M_H(v)$, which is smooth.
Then by using the Fourier-Mukai transform,
we see that ${\cal M}_H(v)^s$ itself is smooth of $\dim {\cal M}_H(v)^s=1$.
\begin{NB}
For unnodal surface, 
there is a Enriques surface parametrizing
$E$ with $\det E \equiv (\rk E/2)K_X  \mod 2$.
By the semi-continuity, 
we have a 2-dimensional component even for a nodal surface.
Then by the theory of Fourier-Mukai transform,
it is smooth of dimension 2.
\begin{NB2}
Let $v$ be an isotropic Mukai vector.
Since the Mukai lattice is unimodular, $M_H(v)$ is a fine moduli space.
Assume that there is an irreducible component $M$
such that $\dim M \geq 2$.
Then $\Hom(E,E(K_X)) \ne 0$ for a general $E \in M$.
By the semi-continuity, $\Hom(E,E(K_X)) \ne 0$ for all $E \in M$.
Since the Zariski tangent space of $M_H(v)$ is
$\Ext^1(E,E) \cong {\Bbb C}^{\oplus 2}$ at $E \in M$,
$M_H(v)$ is smooth of dimension 2 on $M$.
Since $v(E)=v(E(K_X))$, we have $E \cong E(K_X)$ and
$\Ext^2(E,E)={\Bbb C}$.
Then the universal family ${\cal E}$ induces a Fourier-Mukai
trasnform ${\bf D}(X) \to {\bf D}(M)$ by the criterion of Bridgeland.
For $F \in M_H(v)$,
we see that $F \cong {\cal E}_x$, $x \in M$.
Hence if there is $E \in M_H(v)$ with
$\Hom(E,E(K_X))=0$, then
$\dim M_H(v)=1$. 
\end{NB2}
\end{NB}
\end{proof}

We next study the non-primitive case.
We assume that $H$ is a general polarization with respect to
$lv$.
Then $H$ is also a general polarization with
respect to $l' v$ for $1 \leq l' \leq l$.
For $E_0 \in {\cal M}_H(l_0 v)^s$,
we set
\begin{equation}
{\cal J}(l, E_0):=\{E \in {\cal M}_H(l v)^{ss} \mid 
\text{ $E$ is generated by $E_0(p K_X)$, $p \in {\Bbb Z}$ } \},
\end{equation}
where $l_0 \mid l$.
\begin{rem}
If $\ell(v)=2$, then
$E_0(K_X) \cong E_0$ for all $E_0 \in {\cal M}_H(v)^s$, and
if $\ell(v)=1$, then 
$E_0(K_X) \not \cong E_0$ for a general $E_0 \in {\cal M}_H(v)^s$.
\end{rem}

\begin{lem}\label{lem:fiber-dim}
$\dim {\cal J}(l,E_0) \leq -1$.
\end{lem}

\begin{proof}
For $F \in \{E_0(p K_X) \mid p \in {\Bbb Z} \}$, we set
\begin{equation}
{\cal J}(l,E_0, F^{\oplus n}):=\{E \in {\cal J}(l,E_0) \mid
\dim \Hom(F,E)=n \}.
\end{equation}
For $E \in {\cal J}(l,E_0,F^{\oplus n})$,
we have an exact sequence
\begin{equation}\label{eq:ext1}
0 \to \Hom(F,E) \otimes F \to E \to E' \to 0
\end{equation}
and $E' \in {\cal J}(l-nl_0,E_0,F(K_X)^{\oplus n'})$ $(n' \geq 0)$.
Since
${\cal J}(l,E_0,F^{\oplus n})$ is an open substack of the stack of extensions
\eqref{eq:ext1}, 
Lemma \ref{lem:A-est1} implies 
\begin{equation}
\dim {\cal J}(l,E_0,F^{\oplus n}) \leq 
\dim {\cal J}(l-nl_0,E_0,F(K_X)^{\oplus n'})
+nn'-n^2.
\end{equation} 
\begin{NB}
$\dim \Ext^1(E',F)-\dim \Hom(E',F)=\dim \Hom(F(K_X),E')=n'$.
$\dim \Hom(F,E)=n$ is an open condition.
\end{NB}
Then the same proof of \cite[(3.8)]{K-Y} works. 
\end{proof}

\begin{prop}\label{prop:isotropic}
Assume that $X$ is an Enriques surface.
Let $v$ be an isotropic and primitive Mukai vector.
\begin{enumerate}
\item[(1)]
Assume that ${\cal M}_H(l v)^s$ is non-empty. Then
$l=1,2$.
\item[(2)]
${\cal M}_H(2 v,L)^s \ne \emptyset $
if and only if $\ell(v)=1$ and $L \equiv 0 \mod 2$.
Moreover
$$
{\cal M}_H(2 v)^s=\{\varpi_*(F) \mid F \in 
{\cal M}_{\varpi^*(H)}^w(w)^s,\;
\iota^*(F) \not \cong F \},
$$ 
where 
$w=\varpi^*(v)$.
In particular, ${\cal M}_H(l v)^s$ is smooth of
dimension 1.
\item[(3)]
$\dim {\cal M}_H(lv)^{ss} \leq l$.
If $\ell(v)=1$, then
$\dim {\cal M}_H(lv)^{ss} \leq [\frac{l}{2}]$.
\end{enumerate} 
\end{prop}

\begin{proof}
(1)
Since $H$ is general, ${\cal M}_H(lv)^s$ is the same as
the moduli stack of $v$-twisted stable sheaves.
Let $w$ be a primitive and isotropic Mukai vector
of $\widetilde{X}$ with 
$\varpi^*(v)=m w$ $(m \in {\Bbb Z}_{>0})$.
For $E \in {\cal M}_H(l v)^s$,
$\varpi^*(E)$ is $w$-twisted semi-stable with respect to $\varpi^*(H)$.
Indeed by the uniqueness of the Harder-Narasimhan filtration 
of $\varpi^*(E)$, it is $\iota$-invariant.
Since $\varpi$ is \'{e}tale,
it comes from a filtration on $X$.

(a) Assume that $\varpi^*(E)$ is not $w$-twisted stable, and
let $F$ be a $w$-twisted stable proper subsheaf of
$\varpi^*(E)$ with 
\begin{equation}
\frac{\chi(\varpi^*(E),F(n\varpi^*(H)))}{\rk F}=
\frac{\chi(\varpi^*(E),\varpi^*(E)(n\varpi^*(H)))}{\rk \varpi^*(E)},
\;(n \in {\Bbb Z}) .
\end{equation}
Then $\iota^*(F)$ is also a $w$-twisted stable subsheaf of $\varpi^*(E)$ 
with 
\begin{equation}
\frac{\chi(\varpi^*(E),\iota^*(F)(n\varpi^*(H)))}{\rk \iota^*(F)}=
\frac{\chi(\varpi^*(E),\varpi^*(E)(n\varpi^*(H)))}{\rk \varpi^*(E)},
\;(n \in {\Bbb Z}).
\end{equation}
If $\iota^*(F) \cong F$, then
there is a subsheaf $E_1$ of $E$ such that
$F \cong \varpi^*(E_1)$,
which shows $E$ is properly
$v$-twisted semi-stable.
\begin{NB}
$\chi(\varpi^*(E),\varpi^*(*)(n\varpi^*(H)))=2\chi(E,(*)(nH))$.
\end{NB}
Hence $\iota^*(F) \not \cong F$.
We note that $\phi:F \oplus \iota^*(F) \to \varpi^*(E)$ is injective.
Indeed let $G$ be a $w$-twisted stable subsheaf of
$\ker \phi$ with 
\begin{equation}
\frac{\chi(\varpi^*(E),G(n\varpi^*(H)))}{\rk G}=
\frac{\chi(\varpi^*(E),\varpi^*(E)(n\varpi^*(H)))}{\rk \varpi^*(E)},
\;(n \in {\Bbb Z}).
\end{equation}
Then $G \to F$ and $G \to \iota^*(F)$
are isomorphic or zero.
Since $F$ and $\iota^*(F)$ are subsheaves of $\varpi^*(E)$,
we get $G \cong F$ and $G \cong \iota^*(F)$, which is a contradiction.
Therefore $\phi$ is injective.
By the $v$-twisted stability of $E$, $\phi$ is also surjective.
Thus $F \oplus \iota^*(F) \cong \varpi^*(E)$.
Then $\varpi_*(F)^{\oplus 2} \cong E \oplus E(K_X)$ implies
$\varpi_*(F) \cong E \cong E(K_X)$.
By \cite[Lem. 2.3.6]{PerverseII},
$\rk F \leq \rk w$ and the equality holds if and only if
$v(F)=w$.
\begin{NB} 
Since $\varpi^*(H)$ may be non-generic,
we consider $w$-twisted semi-stability.
By \cite[Lem. 2.3.6]{PerverseII}, we still have
$\rk F \leq \rk w$.  
\end{NB}
Hence $lm \rk w =\rk E \leq 2\rk w$, which shows $l m \leq 2$.
Moreover $lm=2$ implies $v(F)=w$.

(b) If $\varpi^*(E)$ is $w$-twisted stable, then
by \cite[Lem. 2.3.6]{PerverseII}, we have
$\rk E \leq \rk w$, which shows $l=m=1$.

(2)
In the proof of (1),
for $E \in {\cal M}_H(2v)^s$, we have 
$\ell(v)=1$ and 
$E=\varpi_*(F)$ with $\iota^*(v(F))=v(F)$.
In particular, $v(F)=w$ and
${\cal M}_H(2v)^s$ is smooth of dimension 1.
\begin{NB}
Indeed if ${\cal M}_H(2v)^s$ is singular at $E$, then
$E=\pi_*(F)$ with $\langle v(F)^2 \rangle=-2$
 and $E$ belongs to a 0-dimensional 
irreducible component ${\cal M}$ of ${\cal M}_H(2v)^s$.
Since $M_{\varpi^*(H)}^w(
If $\langle v(F)^2 \rangle=-2
\end{NB}
By Remark \ref{rem:rho=10} and $4 \mid \rk E$, we have 
$2 \mid c_1(E)$ in $\NS(X)$. 

We note that a primitive and isotropic Mukai vector $v$ 
with $\ell(v)$ corresponds to a primitive and isotropic
Mukai vector $w$ on $\widetilde{X}$ with $\iota^*(w)=w$
via $\varpi^*$.
For such a vector $w$, 
we have $M_{\varpi^*(H)}^w(w)^s \ne \emptyset$ 
(\cite[cf. Cor.1.3.3]{PerverseII}),
and the fixed point set of the $\iota^*$-action  
on $M_{\varpi^*(H)}^w(w)^s$
is 1-dimensional by Lemma \ref{lem:isotropic}.
For $F \in M_{\varpi^*(H)}^w(w)^s$ with 
$\iota^*(F) \not = F$,
$\varpi_*(F)$ is a stable sheaf with respect to $H$.
\begin{NB}
We note that $\varpi^*(\varpi_*(F))=F \oplus \iota^*(F)$ is $w$-twisted
semi-stable and if $E_1 \subset \varpi^* \varpi_*(F)$ is a subsheaf
with $\chi(w,E_1(n \varpi^*(H)))/\rk E_1 \geq 
\chi(w,F(n\varpi^*(H)))/\rk F$
implies $E_1=F,\iota^*(F)$.  
\end{NB}
Therefore (2) holds.

(3)
We have a decomposition
$$
\overline{M}_H(lv)=\bigcup_{(n_1 l_1,...,n_t l_t) \in S_l}
\prod_i S^{n_i} M_H(l_i v),
$$
where
$$
S_l:=\left\{(n_1 l_1,...,n_t l_t) \left| l_1<l_2<\cdots 
<l_t,\; n_1,...,n_t \in {\Bbb Z}_{>0},\;
\sum_i n_i l_i=l \right. \right\}. 
$$
We have a morphism
$\phi:{\cal M}_H(lv)^{ss} \to \overline{M}_H(lv)$.
\begin{NB}
Indeed let $[Q/GL(N)]$ be a description of 
${\cal M}_H(lv)^{ss}$ as a quotient stack. 
Then for a family of semi-stable sheaves ${\cal E}$
over $S$, we have a $GL(N)$-bundle ${\cal P}$
and an equivariant map ${\cal P} \to Q$.
Then ${\cal P} \to Q \to \overline{M}_H(lv)$
is $GL(N)$-invariant.
Hence we have a morphism
$S \to \overline{M}_H(lv)$. 
\end{NB} 
Let $x$ be a point of $\overline{M}_H(lv)$.
Then there are stable sheaves $E_i \in {\cal M}_H(k_i v)^s$
such that $E_i \not \cong E_j(pK_X)$ for $i \ne j$ and
 $x=\oplus_{i=1}^t \oplus_p E_i(pK_X)^{\oplus n_{ip}}$.
Since 
$$
\Hom(E_i,E_j)=\Ext^2(E_i,E_j)=0
$$
for $i \ne j$ and $\chi(E_i,E_j)=0$,
for $E \in \phi^{-1}(x)$,
there are $G_i \in {\cal J}(\sum_p n_{ip}k_i,E_i)$
and $E=\oplus_i G_i$.
By Lemma \ref{lem:fiber-dim} we see that
$\dim \phi^{-1}(x) \leq
-t$.
We note that $\dim M_H(v(E_i))^s=1,2$
by (1), (2) and Lemma \ref{lem:isotropic}.
We set
\begin{equation}
\begin{split}
t_1:= & \{i \mid \dim M_H(v(E_i))=1\},\\
t_2:= & \{i \mid \dim M_H(v(E_i))=2 \}.
\end{split}
\end{equation}
Then we have
$$
\dim {\cal M}_H(lv)^{ss} \leq \max_{x \in \overline{M}_H(lv)}
 \{ -t+t_1+2t_2\} 
=\max_{x \in \overline{M}_H(lv)} t_2 \leq l.
 $$
Moreover if $\ell(v)=1$, then
$\dim M_H(v(E_i))=2$ implies $v(E_i)=2v$,
which implies that 
$t_2 \leq l/2$. Hence the second claim also holds. 
\end{proof}

\begin{NB}
\begin{rem}\label{rem:iso2}
By the proof of Proposition \ref{prop:isotropic} (1),
for $E \in {\cal M}_H(2v)^s$, $\ell(v)=1$ and 
$E=\varpi_*(F)$ with $\iota^*(v(F))=v(F)$.
In particular ${\cal M}_H(2v)^s$ is smooth of dimension 1.
\begin{NB2}
Indeed if ${\cal M}_H(2v)^s$ is singular at $E$, then
$E=\pi_*(F)$ with $\langle v(F)^2 \rangle=-2$
 and $E$ belongs to a 0-dimensional 
irreducible component ${\cal M}$ of ${\cal M}_H(2v)^s$.
Since $M_{\varpi^*(H)}^w(
If $\langle v(F)^2 \rangle=-2
\end{NB2}
By Remark \ref{rem:rho=10}, 
$2 \mid c_1(E)$ in $\NS(X)$.

Conversely for a primitive and isotropic Mukai vector
$w$ on $\widetilde{X}$ with $\iota^*(w)=w$,
$M_{\varpi^*(H)}^w(w)^s \ne \emptyset $ 
by \cite[cf. Cor.1.3.3]{PerverseII},
and the fixed point set of the $\iota^*$-action  
on $M_{\varpi^*(H)}^w(w)^s$
is 1-dimensional by Lemma \ref{lem:isotropic}.
For $F \in M_{\varpi^*(H)}^w(w)^s$ with 
$\iota^*(F) \not = F$,
$\varpi_*(F)$ is a stable sheaf with respect to $H$.
\begin{NB2}
We note that $\varpi^*(\varpi_*(F))=F \oplus \iota^*(F)$ is $w$-twisted
semi-stable and if $E_1 \subset \varpi^* \varpi_*(F)$ is a subsheaf
with $\chi(w,E_1(n \varpi^*(H)))/\rk E_1 \geq 
\chi(w,F(n\varpi^*(H)))/\rk F$
implies $E_1=F,\iota^*(F)$.  
\end{NB2}
\end{rem}
\end{NB}

\begin{rem}\label{rem:isotropic}
Let $\pi:X \to C$ be an elliptic surface and
$mD$ be a tame multiple fiber.
Let $v:=(0,rD,d)$ be a primitive Mukai vector,i.e., $\gcd(r,d)=1$.
For a semi-stable sheaf $E$ with
$v(E)=lv$ and $\Div(E)=lrD$,
we shall show in Lemma \ref{lem:relative-moduli} that
$E$ is $S$-equivalent to $\oplus_i E_i$,
where $E_i \in {\cal M}_H(v)^s$.
Assume that $m \nmid r$. Then 
$E_i \otimes K_X \not \cong E_i$, which implies 
$\dim {\cal M}_H(v)^s=0$. 
Hence we see that
$\dim {\cal M}_H(lv)^{ss} \leq [\frac{l m_0}{m}]$,
where $m_0=\gcd(r,m)$.
\end{rem}

\section{$\mu$-stability}
In this section, we continue to assume that $X$ is an Enriques surface, and
we shall study the existence condition of $\mu$-stable
locally free sheaves.
For a Mukai vector $v$ of $\rk v>0$, we have a decomposition
$v=(lr,l\xi,\frac{s}{2})$, where 
$\gcd(r,\xi)=1$, $l \in {\Bbb Z}_{>0}$,
$s \in {\Bbb Z}$, 
$lr-s \in 2{\Bbb Z}$.
We devide into three cases:  
\begin{enumerate}
\item[A.]
There is no stable sheaf $E$ such that 
$v(E)=(r,\xi,b)$ and 
$\langle v(E)^2 \rangle=-1,-2$.
\item[B.]
There is a stable sheaf $E$ such that 
$v(E)=(r,\xi,b)$ and  
$\langle v(E)^2 \rangle=-1$.
\item[C.]
There is a stable sheaf $E$ such that 
$v(E)=(r,\xi,b)$ and $\langle v(E)^2 \rangle=-2$.
\end{enumerate}
\begin{rem}
$r$ is odd for case B and
$r$ is even for case C. 
\end{rem}

By a case by case study,
we shall prove the following result. 
\begin{thm}\label{thm:mu-s}
Let $v=(lr,l\xi,\frac{s}{2})$ be a Mukai
vector such that $\gcd(r,\xi)=1$ and $\langle v^2 \rangle \geq 0$.
Let $H$ be a general polarization with respect to $v$.
Then for $L \in \NS(X)$ with $[L \mod K_X] =l \xi$,
${\cal M}_H(v,L)^{ss}$ contains a $\mu$-stable
sheaf if and only if 
\begin{enumerate}
\item[A.] 
There is no stable sheaf $E$ such that 
$v(E)=(r,\xi,b)$ and 
$\langle v(E)^2 \rangle=-1,-2$
and,
$\langle v^2 \rangle \geq 0$ or
\item[B.] 
There is a stable sheaf $E$ such that 
$v(E)=(r,\xi,b)$ and  
$\langle v(E)^2 \rangle=-1$, and
$\langle v^2 \rangle \geq l^2$ or
\item[C.] 
There is a stable sheaf $E$ such that 
$v(E)=(r,\xi,b)$ and $\langle v(E)^2 \rangle=-2$, and
$\langle v^2 \rangle \geq 2l^2$. 
\end{enumerate}
Moreover if $lr>1$, then
under the same condition,
${\cal M}_H(v,L)^{ss}$ contains a $\mu$-stable locally free
sheaf.
\end{thm}

Although the arguments in this section are similar to \cite{Y:irred},
we repeat the arguments since several estimates are slightly different.
Throughout this section, $H$ is a general polarization
with respect to $v$. 

\subsection{Case A}
In this subsection, we shall treat case A.
Let
$v:=l(r+\xi)+a\varrho_X \in H^*(X,{\Bbb Q})$
be a Mukai vector.
We shall first estimate the dimension of various 
locally closed substacks of
${\cal M}(v)$.

\begin{lem}\label{lem:non-empty}
If ${\cal M}_H(v)^{\mu ss} \ne \emptyset$, then 
$\langle v^2 \rangle \geq 0$.
If the equality holds, then ${\cal M}_H(v)^{\mu ss}={\cal M}_H(v)^{ss}$
and $E \in {\cal M}_H(v)^{ss}$ is $S$-equivalent to
$\oplus_i E_i$, where $E_i$ are $\mu$-stable locally free sheaves wth 
$v(E_i) \in {\Bbb Q}v$.
\end{lem}
\begin{proof}
Let $E$ be a $\mu$-semi-stable sheaf of $v(E)=v$
and choose a Jordan-H\"{o}lder filtration of
$E$ with respect to $\mu$-stability
whose factors are $\mu$-stable sheaves $E_i$ $(1 \leq i \leq s)$.
We set
\begin{equation}
v(E_i):=l_i(r+\xi)+a_i\varrho_X, 1 \leq i \leq s.
\end{equation}
By our assumption, 
$\langle v(E_i)^2 \rangle=l_i(l_i(\xi^2)-2r a_i) \ne -1,-2$.
\begin{NB}
If $\langle v(E_i)^2 \rangle=-2$, then $a_i$ is even, which implies that
$l_i$ is 1.
\end{NB}
Thus $\langle v(E_i)^2 \rangle \geq 0$ for all $i$.
Since
\begin{equation}\label{eq:(2.2)}
\frac{\langle v^2 \rangle}{l}=
\sum_{i=1}^s\frac{\langle v(E_i)^2 \rangle}{l_i},
\end{equation}
we get $\langle v^2 \rangle \geq 0$.
Assume that $\langle v^2 \rangle=0$. 
Then 
$\langle v(E_i)^2 \rangle =0$ for all $i$.
Since 
\begin{equation}\label{eq:(2.3)}
\begin{split}
\frac{\langle v(E_i)^2 \rangle}{\rk(E_i)^2}=& (\xi^2)-2 \frac{a_i}{r l_i},\\
\frac{\chi(E_i)}{\rk(E_i)}=& \frac{1}{2}+\frac{a_i}{rl_i},
\end{split}
\end{equation}
we see that $\chi(E_i)/\rk(E_i)=\chi(E)/\rk(E)$ for all $i$.
Thus $E$ is semi-stable.
Since $E_i^{\vee \vee}$ are $\mu$-stable locally free sheaves
with $\langle v(E_i^{\vee \vee})^2 \rangle \geq 0$
and $0=\langle v(E_i)^2 \rangle =\langle v(E_i^{\vee \vee})^2 \rangle
+2\rk E_i \chi(E_i^{\vee \vee}/E_i)$, we see that
$\chi(E_i^{\vee \vee}/E_i)=0$. Thus 
all $E_i$ are locally free, which shows that $E$ is also locally free. 
\end{proof}
\begin{cor}\label{cor:isotropic}
If $\langle v^2 \rangle=0$, then ${\cal M}_H(v)^{\mu ss}$ 
consists of locally free sheaves.
\end{cor}

\begin{defn}
Let $w=l_0(r+\xi)+a_0 \varrho_X$ $(l_0>0)$ be the primitive Mukai vector  
such that $\langle w^2 \rangle =0$.
Since $a_0/l_0=(\xi^2)/(2r)$,
$w$ is uniquely determined.
\end{defn}

By Lemma \ref{lem:non-empty} and Corollary \ref{cor:isotropic},
 ${\cal M}_H(w)^{\mu ss}$ consists of $\mu$-stable locally free sheaves.
Since $l_0 (\xi^2)-2a_0 r=0$,
$l_0 r$ is even. In particular, $a_0 \in {\Bbb Z}$.

\begin{lem}\label{lem:w-v}
For a Mukai vector $u=(l r,l \xi,a)$,
$r \mid \langle u,w \rangle$.
\end{lem}

\begin{proof}
We note that $l_0 r$ is even and $a_0 \in {\Bbb Z}$.
If $r$ is even, then $a \in {\Bbb Z}$.
If $r$ is odd, then $l_0$ is even.
Hence $l_0 a \in {\Bbb Z}$.
Then $\langle u,w \rangle=
(la_0-l_0 a)r$ is divisible by $r$. 
\end{proof}

\begin{lem}\label{lem:stable}
\begin{enumerate}
\item[(1)]
\begin{equation}
\dim ({\cal M}_H(v)^{\mu ss} \setminus {\cal M}_H(v)^{ss})
\leq \langle v^2 \rangle-1.
\end{equation}
\item [(2)]
Assume that $\langle v^2 \rangle>0$. Then
\begin{equation}
\dim ({\cal M}_H(v)^{\mu ss} \setminus {\cal M}_H(v)^{s})
\leq \langle v^2 \rangle-1.
\end{equation}
In particular,
if ${\cal M}_H(v)^{\mu ss} \ne \emptyset$, then
${\cal M}_H(v)^s \ne \emptyset$ and
$\dim {\cal M}_H(v)^{\mu ss}=\langle v^2 \rangle$.
\end{enumerate}
\end{lem}
\begin{proof}
By Lemma \ref{lem:pss}, it is sufficient to prove (1).
Let $F$ be a $\mu$-semi-stable sheaf of $v(F)=v$.
We assume that $F$ is not semi-stable.
Let 
\begin{equation}
 0 \subset F_1 \subset F_2 \subset \dots \subset F_s=F
\end{equation}
be the Harder-Narasimhan filtration of $F$.
We set 
\begin{equation}
 v_i:=v(F_i/F_{i-1})=l_i(r+\xi)+a_i \varrho_X, 1 \leq i \leq s.
\end{equation}
Since $\chi(F_i/F_{i-1})/\rk(F_i/F_{i-1})>
\chi(F_{i+1}/F_{i})/\rk(F_{i+1}/F_{i})$,
we get that
\begin{equation}\label{eq:a*}
 \frac{a_0}{l_0} \geq \frac{a_1}{l_1}>\frac{a_2}{l_2}> \dots
 >\frac{a_s}{l_s},
\end{equation}
where the leftmost inequality is a consequence of
$\langle w^2 \rangle=0$, $\langle v_i^2 \rangle \geq 0$ and 
\eqref{eq:(2.3)}.
Let ${\cal F}^{HN}(v_1,v_2,\dots,v_s)$ be the substack of 
${\cal M}_H(v)^{\mu ss}$
whose element $E$ has the Harder-Narasimhan filtration
of the above type.
We shall prove that
$\dim {\cal F}^{HN}(v_1,v_2,\dots,v_s) \leq \langle v^2 \rangle-1$.
Since 
$$
\Hom(F_i/F_{i-1},F_j/F_{j-1}(K_X))=0
$$
 for $i<j$,
\cite[Lem. 5.3]{K-Y} implies that
\begin{equation}
  \dim {\cal F}^{HN}(v_1,v_2,\dots,v_s)=
  \sum_{i=1}^s \dim {\cal M}_H(v_i)^{ss}+
  \sum_{i<j}\langle v_j,v_i \rangle.
\end{equation}
For $i<j$, by using Lemma \ref{lem:non-empty} and \eqref{eq:a*},
we see that
\begin{equation}\label{eq:v_i,v_j}
\begin{split}
\langle v_i,v_j \rangle &=
l_il_j (\xi^2)-(l_ia_j+l_j a_i)r\\
&=l_il_j (\xi^2)-2l_ja_i r+
(a_il_j-a_jl_i)r\\
&=l_j(l_i (\xi^2)-2a_ir)+(a_il_j-a_jl_i)r\\
& \geq (a_il_j-a_jl_i)r \geq r/2 \geq 1,
\end{split}
\end{equation}
where the inequality $r \geq 2$ comes from our assumption.
Hence if $\langle v_i^2 \rangle>0$ for all $i$,
then, by using Lemma \ref{lem:pss},
we see that
\begin{equation}
 \dim {\cal F}^{HN}(v_1,v_2,\dots,v_s) =
 \langle v^2 \rangle-
 \sum_{i<j}\langle v_i,v_j \rangle 
 \leq \langle v^2 \rangle-1.
\end{equation} 
Assume that $\langle v_i^2 \rangle =0$, i.e. $v_i=l_i'w$, $l_i \in {\Bbb Z}$.
Then $i=1$ and 
$r \mid \langle v_j,w \rangle$ by Lemma \ref{lem:w-v}.
Hence 
\begin{equation}\label{eq:v_1,v_j}
 \begin{split}
  \langle v_1,v_j \rangle-l_1' &=l_1'(\langle w,v_j \rangle-1)\\
  & \geq l_1'(r-1)>0.
 \end{split}
\end{equation}
In this case, by using Proposition \ref{prop:isotropic}, we see that
\begin{equation}
 \dim {\cal F}^{HN}(v_1,v_2,\dots,v_s) \leq
 \langle v^2 \rangle-
 \left(\sum_{i<j}\langle v_i,v_j \rangle-l_1'\right)
 \leq \langle v^2 \rangle-1.
\end{equation} 
Hence we get our lemma.
\end{proof}

\begin{prop}\label{prop:stable}
Assume that $\langle v^2 \rangle >0$.
Then for a general $H$,
\begin{equation}
\dim( {\cal M}_H(v)^s \setminus {\cal M}_H(v)^{\mu s})
< \langle v^2 \rangle.
\end{equation}
In particular,
${\cal M}_H(v)^{\mu s} \ne \emptyset$. 
Moreover 
there is a $\mu$-stable locally free sheaf in
each connected component.
\end{prop}
\begin{proof}
Let $E$ be a stable sheaf with $v(E)=v$ and 
$E_1$ be a $\mu$-stable subsheaf of 
$E$ such that
$E/E_1$ is torsion free.
We set
\begin{equation}
 \begin{split}
  v_1:=v(E_1)=(l_1 r, l_1 \xi,a_1),\\
  v_2:=v(E/E_1)=(l_2 r, l_2 \xi,a_2).
 \end{split}
\end{equation}
Since $\chi(E_1)/\rk E_1<\chi(E)/\rk E$,
we get $\langle v(E_1)^2 \rangle>0$ and 
\begin{equation}\label{eq:stable}
 \frac{a_1}{l_1}<\frac{a_2}{l_2}.
\end{equation}
Let $J(v_1,v_2)$ be the substack  of
${\cal M}_H(v)^s$ consisting of $E$ which has a subsheaf
$E_1 \subset E$.
By using Lemma \ref{lem:A-est1},
 we shall estimate
$\dim J(v_1,v_2)$.
By \cite[Lem. 3.1]{K-Y},
$ \dim \Hom(E_1,(E/E_1)(K_X)) \leq l_2/l_1$.
We shall bound the dimension of the substack
\begin{equation}
 {\cal N}_n(v_1,v_2):=\{(E_1, E_2) \in {\cal M}_H(v_1)^{\mu s} \times
 {\cal M}_H(v_2)^{\mu ss} \mid \dim (E_1^{\vee \vee}/E_1)=n,\;
\dim \Hom(E_1,E_2(K_X)) \ne 0 \}.
\end{equation}
For a fixed $E_2 \in {\cal M}_H(v_2)^{\mu ss}$,
\begin{equation}
 \#\{E_1^{\vee \vee}|
 E_1 \in {\cal M}_H(v_1)^{\mu s}, \Hom(E_1,E_2(K_X)) \ne 0 \} <\infty.
\end{equation}
Indeed the double dual of the graded object associated to the
Jordan H\"{o}lder filtration with respect to $\mu$-stability
is well defined and $E_1^{\vee \vee}$ must be one of these stable factors.
For a locally free sheaf,  
let 
$\Quot_{F/X}^n$
be the quot-scheme parametrizing all quotients $F \to A$
such that $A$ is a 0-dimensional sheaf of $\chi(A)=n$.
In \cite[Thm. 0.4 and sect. 5]{Y:1}, we computed the number of rational points
of $\Quot_{F/X}^n$ over finite fields, which implies that
\begin{equation}\label{eq:quot}
\dim \Quot_{F/X}^n=(\rk F+1)n.
\end{equation}
Since $E_1 \in {\cal M}_H(v_1)^{\mu s}$ is simple, 
we see that
\begin{equation}
 \dim \{E_1 \in {\cal M}_H(v_1)^{\mu s}|\dim (E_1^{\vee \vee}/E_1)=n,
\Hom(E_1,E_2(K_X)) \ne 0 \}
 \leq (\rk v_1+1)n-1.
\end{equation}
Since $\dim {\cal M}_H(v_1)^{ss}=
\langle v_1^2 \rangle \geq 2l_1 rn$ and $r \geq 2$,
we get 
\begin{equation}\label{eq:estimate-N}
\begin{split}
\dim {\cal N}_n(v_1,v_2) \leq & \dim{\cal M}_H(v_1)^{\mu s}+
\dim {\cal M}_H(v_2)^{\mu ss}-((l_1 r-1)n+1)\\
\leq & \dim{\cal M}_H(v_1)^{\mu s}+
\dim {\cal M}_H(v_2)^{\mu ss}-2
\end{split}
\end{equation}
if $n>0$.
If $n=0$, then
the same inequality also holds, since 
$\langle v_1^2 \rangle>0=2l_1 rn$.
Moreover, if $\langle v_2^2 \rangle=0$, then
Lemma \ref{lem:non-empty} implies 
$\Hom(E_1,E_2) \ne 0$ implies $v(E_1^{\vee \vee}) \in {\Bbb Q}v_2$, which
shows $l_1=l_0$.
Therefore 
${\cal N}_n(v_1,v_2)=\emptyset$ unless $l_1=l_0$.

If $\langle v_2^2 \rangle >0$, then
Lemma \ref{lem:stable} implies that
$\dim {\cal M}_H(v_2)^{\mu ss}=\langle v_2^2 \rangle$.
Hence 
Lemma \ref{lem:A-est1} implies that 
\begin{equation}
 \begin{split}
   \dim {\cal M}_H(v)^s-\dim J(v_1,v_2)&=
   \min \left(\langle v_1,v_2\rangle-\frac{l_2}{l_1}+2,
  \langle v_1,v_2\rangle \right)\\
   &=l_1 \frac{\langle v_2^2 \rangle}{2l_2}+
   l_2\frac{\langle v_1^2 \rangle}{2l_1}-
   \max\left(\frac{l_2}{l_1}-2,0 \right)>0.
 \end{split}
\end{equation}
We next treat the case where $\langle v_2^2 \rangle=0$.
Then $v_2=l_2' w, l_2' \in {\Bbb Z}$.
By Proposition \ref{prop:isotropic},
$\dim {\cal M}_H(v_2)^{\mu ss} \leq \langle v_2^2 \rangle+l_2'$.
If $l_1=l_0$, then $l_2/l_1=l_2'$. 
In this case, by using \eqref{eq:estimate-N} and Lemma \ref{lem:w-v},
we see that
\begin{equation}
\begin{split}
   \dim {\cal M}_H(v)^s-\dim J(v_1,v_2) \geq &
\langle v_1,v_2 \rangle-l_2/l_1-l_2'+2\\
\geq & l_2'(\langle v_1,w \rangle-2)+2 >0.
\end{split}
\end{equation}
If $l_1 \ne l_0$, then since ${\cal N}_n(v_1,v_2) =\emptyset$,
we see that
\begin{equation}
   \dim {\cal M}_H(v)^s-\dim J(v_1,v_2)=
l_2'(\langle v_1,w \rangle-1) >0
\end{equation}
by Lemma \ref{lem:w-v}.

We shall prove that there is a $\mu$-stable locally free sheaf.
Let ${\cal M}_H(v)^{nlf}$ be the closed substack of 
${\cal M}_H(v)^{\mu s}$ consisting of non-locally free
sheaves.
By \eqref{eq:quot},
we have 
\begin{equation}\label{eq:nlf}
\dim {\cal M}_H(v)^{nlf} \leq \max_{b>0}
(\dim {\cal M}_H(v+b \varrho_X)^{\mu s}+(rl+1)b)
\leq \langle v^2 \rangle+\delta-(rl-1),
\end{equation}
where 
\begin{equation}
\delta=
\begin{cases}
\tfrac{l}{l_0},& v+b\varrho_X=\tfrac{l}{l_0} w, \ell(w)=2,\\
[\tfrac{l}{2l_0}],& v+b\varrho_X=\tfrac{l}{l_0} w, \ell(w)=1,\\
0, & \text{ otherwise.}
\end{cases}
\end{equation}
If $\ell(w)=1$, then
$$
(rl-1)-\delta \geq \tfrac{l}{l_0}(rl_0-1/2)-1>0,
$$
since $rl_0$ is even.
If $\ell(w)=2$, then
$$
(rl-1)-\delta \geq \tfrac{l}{l_0}(rl_0-1)-1>0,
$$
unless $rl_0=2$, $l=l_0$.
In this case, we see that $r=1$ and $w=2 e^{\xi}$. 
Since $u:=(1,\xi,\frac{(\xi^2)+1}{2})$ satisfies
$\langle u^2 \rangle=-1$, this case does not occur.
Therefore $\dim {\cal M}_H(v)^{nlf}<\langle v^2 \rangle$
and the last claim holds. 

\end{proof}

\subsection{Case B}

Assume that 
$r \mid (\xi^2)+1$.
Then 
$v_0:=(r,\xi,a_0)$ is a primitive Mukai vector with
$\langle v_0^2 \rangle=-1$,
where $a_0:=\frac{(\xi^2)+1}{2r} \in {\Bbb Z}+\frac{1}{2}$.
We take a general polarization $H$ with respect to $v_0$.
Let $F$ be the $\mu$-stable locally free sheaf with $v(F)=v_0$.

\begin{lem}\label{lem:l^2}
If there is a $\mu$-stable sheaf $E$ with 
$v(E)=l v_0+b \varrho_X$, then
$\langle v(E)^2 \rangle \geq l^2$
or $E \cong F, F(K_X)$.
\end{lem}

\begin{proof}
If $\rk E=r$, then
$l=1$ and $b \leq 0$, which implies the claim.
Assume that $\rk E>r$.
By the $\mu$-stability of $E,F$,
$\Hom(F,E)=\Hom(E,F(K_X))=0$.  
Hence $0 \geq \chi(F,E)=-\langle v_0,v(E) \rangle=
l+br$.
Since $\langle v(E)^2 \rangle=l(-l-2br)$,
we get the claim.
\end{proof}

\begin{rem}
Since $l=\gcd(\rk E,c_1(E)) \in {\Bbb Z}$, 
we have $v(E), lv_0 \in v(K(X))$, which implies $b \in {\Bbb Z}$.
\end{rem}

\begin{NB}
Lt seems that we don't need the following 15/10/4:
If $l_i$ is odd, then 
$\Ext^2(E_i,E_i)=0$.
If $l_i$ is even, then
by Lemma \ref{lem:l^2} and Lemma
\ref{lem:yamada},
$\dim {\cal M}_H(v_i)_{sing}^s+\dim \Ext^2(E_i,E_i)
\leq \langle v_i^2 \rangle$.
\end{NB}

\begin{lem}[cf. {\cite[Lem. 4.4]{Y:K3}}]\label{lem:quot}
Let $v$ be an arbitrary Mukai vector of $\rk v>0$.
Let ${\cal M}_H(v)^{\mu ss}$ be the moduli stack of 
$\mu$-semi-stable sheaves $E$ of $v(E)=v$, and
${\cal M}_H(v)^{p \mu ss}$ the closed substack of 
${\cal M}_H(v)^{\mu ss}$ consisting
of properly $\mu$-semi-stable sheaves.
We assume that $\langle v^2 \rangle \geq l^2$.
Then
\begin{enumerate}
\item[(1)]
\begin{equation}
\langle v^2 \rangle-
\dim {\cal M}_H(v)^{p \mu ss} \geq \langle v^2 \rangle/2l-l/2+1
\end{equation}
unless $r=1$ and $l=2$.
\item[(2)]
If ${\cal M}_H(v)^{\mu ss}$ is not empty,
then there is a $\mu$-stable locally free sheaf $E$ of $v(E)=v$
in each connected component,
unless $v=(2,0,-1)e^\xi$.   
\item[(3)]
If $v=(2,0,-1)e^\xi$, then
there is an irreducible component of ${\cal M}_H(v,L)^{ss}$ containing 
$\mu$-stable locally free sheaves. 
\end{enumerate}
\end{lem}

\begin{proof}
By Lemma \ref{lem:stack-dim}, we get that
\begin{equation}
\dim {\cal M}_H(v)^{\mu ss} \geq \langle v^2 \rangle. 
\end{equation}
We shall show that 
\begin{equation}
\dim {\cal M}_H(v)^{p \mu ss} \leq \langle v^2 \rangle -
(\langle v^2 \rangle/2l-l/2+1).
\end{equation}
For this purpose, we shall estimate the moduli number of
Jordan-H\"{o}lder filtrations. 
Let $E$ be a $\mu$-semi-stable sheaf of
$v(E)=v$ and let
\begin{equation}\label{eq:mu-JHF}
0 \subset F_1 \subset F_2 \subset \cdots \subset F_t=E
\end{equation}
be a Jordan-H\"{o}lder filtration of $E$ with respect to $\mu$-stability.
We set $E_i:=F_i/F_{i-1}$.
We also set
$$
v(E):=l v_0+a \varrho_X,
v(E_i):=l_i v_0+a_i \varrho_X.
$$
By Lemma \ref{lem:l^2},
$\langle v_i^2 \rangle \ne 0$.
Hence $\dim {\cal M}_H(v_i)^{\mu s}=\langle v_i^2 \rangle$.
Let $J(v_1,v_2,...,v_t)$ be the substack of ${\cal M}_H(v)^{\mu ss}$
such that $E \in {\cal M}_H(v)^{\mu ss}$ has a filtration \eqref{eq:mu-JHF}.
By using 
Lemma \ref{lem:A-est1} successively,
we see that
\begin{equation}
\begin{split}
\dim J(v_1,v_2,...,v_t) \leq 
& \sum_i \dim {\cal M}_H(v_i)^{\mu s}+
\sum_{i < j} (\dim \Ext^1(E_j,E_i)-\dim \Hom(E_j,E_i))\\
=&  -\chi(E,E)+\sum_{i>j}\chi(E_j,E_i)+
\sum_{i<j}\dim \Ext^2(E_j,E_i).
\end{split}
\end{equation}
Since $\langle v(E_i),v(E_j) \rangle=-l_i l_j -r(l_i a_j+l_j a_i)$,
we see that
\begin{equation}
\sum_{i>j}\chi(E_j,E_i) = -\sum_{i>j} \langle v(E_j),v(E_i) \rangle
=-\sum_{i}\frac{(l-l_i)\langle v(E_i)^2 \rangle}{2l_i}.
\end{equation}
We set $\max_i \{l_i\}=(l-k)$.
If $\langle v(E_i)^2 \rangle=-1$ for all $i$, then
$v(E_i)=v_0$ for all $i$. Hence 
there is an integer $i_0$ such that $ \langle v(E_{i_0})^2 \rangle \geq 0$.
Since $(l-k)+(t-1) \leq \sum_i l_i=l$, we obtain that $t \leq k+1$.  
Since $l-l_i-k \geq 0$ and $\langle v(E_i)^2 \rangle \geq -1$,
 we get that
\begin{align*}
\sum_{i>j} \langle v(E_j),v(E_i) \rangle 
&=k\sum_i \frac{\langle v(E_i)^2 \rangle}{2l_i}+
\sum_i \frac{(l-l_i-k)\langle v(E_i)^2 \rangle}{2l_i}\\
&=k \frac{\langle v(E)^2 \rangle}{2l}+
\sum_i \frac{(l-l_i-k)\langle v(E_i)^2 \rangle}{2l_i}\\
& \geq  k\frac{\langle v(E)^2 \rangle}{2l}-
\sum_{i \ne i_0} \frac{(l-l_i-k)}{2}\\
& \geq k\frac{\langle v(E)^2 \rangle}{2l}-\frac{(l-1-k)k}{2}.
\end{align*}
Assume that 
$\Ext^2(E_j,E_i)=0$ for some $i<j$.
Then 
we get that 
$$
\sum_{i<j} \dim \Ext^2(E_j,E_i) \leq \frac{(k+1)k}{2}-1.
$$
\begin{NB}
Old version:
If $r>1$ or $l_i>1$ for some $i$,
then for a general filtration, there are $E_i$ and $E_j$ such that
$\Ext^2(E_j,E_i)=0$.
Therefore we get that 
$\sum_{i<j} \dim \Ext^2(E_j,E_i) \leq (k+1)k/2-1$
for a general filtration.
\end{NB}
Then the moduli number of these filtrations is bounded by
\begin{equation}\label{eq:l-estimate}
\langle v^2 \rangle- k\frac{\langle v^2 \rangle}{2l}
+\frac{(l-1-k)k}{2}+\frac{(k+1)k}{2}-1 
\leq 
\langle v^2 \rangle-1- k\left(\frac{\langle v^2 \rangle}{2l}
-\frac{l}{2}\right).
\end{equation}
Therefore we get a desired estimate for this case.

Assume that 
$\Ext^2(E_j,E_i) \cong {\Bbb C}$ for all $i<j$.
Then $E_i^{\vee \vee} \cong E_j^{\vee \vee}(K_X)$.
In particular,
$l_i=l_j$ for all $i<j$.
Suppose first that $l_i \geq 2$ for all $i$, then
$l_i=l-k$ implies $k \leq l-2$ and
$$
\sum_{i>j} \langle v_i,v_j \rangle=
k \frac{\langle v^2 \rangle}{2l}.
$$
Since $4 \leq l_1+l_2=2(l-k) \leq \sum_i l_i=l$,
we have $2 \leq l/2 \leq k$.  
Hence the dimension of these filtrations is bounded by
\begin{equation}\label{eq:moduli-number1}
\begin{split}
\langle v^2 \rangle- k\frac{\langle v^2 \rangle}{2l}
+\frac{(k+1)k}{2} 
\leq &
\langle v^2 \rangle- k\left(\frac{\langle v^2 \rangle}{2l}
-\frac{l-1}{2}\right)\\
\leq & \langle v^2 \rangle- \left(\frac{\langle v^2 \rangle}{2l}
-\frac{l}{2}+1\right).
\end{split}
\end{equation}
 
If instead $l_i=1$ for all $i$,
then we must have $t=2$, and hence $l=2$.
Indeed, if we have $h<i<j$, then
$E_h^{\vee \vee} \cong E_i^{\vee \vee} (K_X)$,
$E_h^{\vee \vee} \cong E_j^{\vee \vee} (K_X)$ and
$E_i^{\vee \vee} \cong E_j^{\vee \vee} (K_X)$, from which it follows
that 
$E_i^{\vee \vee} \cong E_i^{\vee \vee}(K_X)$.
As $\rk E_i$ is odd, this is impossible, so $t=2$ as claimed, and
$l=2$ follows from this and $l_i=1$.

Assume that $r>1$. Then a general member $E_1 \in {\cal M}_H(v_1)^{\mu s}$
is locally free by a similar estimate to \eqref{eq:nlf}.
If there is a non-zero homomorphism $\phi:E_1 \to E_2(K_X)$, then
$\phi$ is injective and $\coker \phi$ is of 0-dimensional
by the $\mu$-stability of $E_1$ and $E_2$.
Since $E_1$ is locally free, $E_1 \cong E_2(K_X)$.
In particular $v_1=v_2$ and $v=2v_1$.
Since $\langle v^2 \rangle \geq l^2$,
$\langle v_1^2 \rangle=\langle v_2^2 \rangle>0$.
Then for a general locally free sheaf $E_1$, 
we have $\Hom(E_1,E_2(K_X)) =0$.
Therefore for a general filtration,
we have $\Ext^2(E_2,E_1)=0$, which shows the same estimate of
\eqref{eq:l-estimate} holds.
Therefore (1) holds.

(2)
The existence of a locally free sheaf follows from 
Lemma \ref{lem:l^2} and the 
last paragraph of the proof of
Proposition \ref{prop:stable} unless $r=1$ and $l=2$.
So we assume that $r=1$ and $l=2$.
This case is treated by Kim \cite{Kim4}.
For completeness, we give a different argument.
If $E_2$ is not locally free or
$\det E_1=\det E_2$, then $\Ext^2(E_2,E_1)=0$ for a 
general filtration, and hence 
the same estimate of \eqref{eq:l-estimate} holds.
On the other hand, if
$E_2$ is locally free and
$E_1=I_Z \otimes E_2(K_X)$ (which implies $E$ is not locally free), 
then
we only have the estimate
\begin{equation}
\langle v^2 \rangle- k\frac{\langle v^2 \rangle}{2l}
+\frac{(l-1-k)k}{2}+\frac{(k+1)k}{2} 
\leq 
\langle v^2 \rangle- \left(\frac{\langle v^2 \rangle}{2l}
-\frac{l}{2}\right).
\end{equation}
In this case, if $\langle v^2 \rangle>4$, 
then there is a $\mu$-stable locally free sheaf.

(3)
We set $v:=(2,0,-1)$.
${\cal M}_H(v,0)^{ss}$ contains a $\mu$-stable locally free sheaf
by the proof of (2).
Indeed we have $\det E_1=\det E_2$, 
which shows \eqref{eq:l-estimate}.
We next treat  
${\cal M}_H(v,K_X)^{ss}$.
By the proof of (2),
it is sufficient to construct a $\mu$-semi-stable locally free
sheaf $E$ of $v(E)=v$ and $\det E={\cal O}_X(K_X)$.
Indeed, for an irreducible component containing a locally
free sheaf, $E_1$ is a locally free sheaf and there is an
ideal sheaf of two points with
$E_2=E_1(K_X) \otimes I_{Z'}$,
which shows \eqref{eq:l-estimate}.

For the ideal sheaf $I_Z$ of two points,
we have 
$\Hom(I_Z(K_X),{\cal O}_X)=\Ext^2(I_Z(K_X),{\cal O}_X)=0$.
Hence $\Ext^1(I_Z(K_X),{\cal O}_X) \cong {\Bbb C}$.
We take a non-trivial extension
\begin{equation}\label{eq:(2,0,-1)}
0 \to {\cal O}_X \to E \to I_Z(K_X) \to 0.
\end{equation}
Since $\Ext^1(I_W(K_X),{\cal O}_X)=0$ for $c_2(I_W)=0,1$,
if $E$ is not locally free, then
$E^{\vee \vee} \cong {\cal O}_X \oplus {\cal O}_X(K_X)$,
which shows that the exact sequence \eqref{eq:(2,0,-1)} splits.
Therefore $E$ is locally free.

\end{proof}

\begin{rem}
For surjective homomorphisms
$\phi_1:{\cal O}_X \to {\Bbb C}_x \oplus {\Bbb C}_y$
and 
$\phi_2:{\cal O}_X(K_X) \to {\Bbb C}_x \oplus {\Bbb C}_y$,
the kernel $E$ of 
$$
{\cal O}_X \oplus {\cal O}_X(K_X) 
\overset{(\phi_1,\phi_2)}{\longrightarrow}
 {\Bbb C}_x \oplus {\Bbb C}_y
$$
is a stable non-locally free sheaf.
Then they form an irreducible component of  
${\cal M}_H(v,K_X)^{ss}$ consisting of non-locally free sheaves.
Therefore ${\cal M}_H(v,K_X)^{ss}$ has at least two irreducible
components.
Combining \cite[Rem. 4.1]{Y:twist1},
${\cal M}_H(2v_0,L)^{ss}$ are reducible
if $\langle v_0^2 \rangle=1$.
\begin{NB}
We note that $\Ext^1({\cal O}_X,I_Z(K_X))={\Bbb C}^{\oplus 2}$.
Since $\Hom(I_Z(K_X),{\cal O}_X)=0$,
$E$ fitting in the extension
$$
0 \to I_Z(K_X) \to E \to {\cal O}_X \to 0
$$
is determined by ${\Bbb P}(\Ext^1({\cal O}_X,I_Z(K_X)))$.
Hence we have a 5-dimensional family of stable sheaves.

Another argument:
For $\phi_i=(\alpha_i,\beta_i)$,
$\ker(\phi_1,\phi_2)$ is parametrized by 
$(\beta_1/\alpha_1:\beta_2/\alpha_2)$.
Hence we have a 1-dimensional moduli if we fixed $x,y \in X$.  
\end{NB}

\end{rem}

By the reflection associated to $v_0$ (cf. \cite{Y:twist1}), 
we get the following result.
\begin{prop}
Assume that
$2br-l>0$, i.e.,
$\langle(lv_0-b \varrho_X)^2 \rangle=l(2br-l)>0$.
Then  
$$
{\cal M}_H(lv_0-b \varrho_X)^{ss} \cong
{\cal M}_H((2br-l)v_0^{\vee}-b \varrho_X)^{ss}.
$$
\end{prop}

\subsection{Case C}

Assume that there is a stable sheaf $E_0$ such that
$v(E_0)=(r,\xi,a_0)$ and $\langle v(E_0)^2 \rangle=-2$.
Thus we assume that 
$r$ is even, 
$r \mid (\xi^2)/2+1$ and
$\xi \equiv D+\frac{r}{2}K_X \mod 2$, where
$D$ is a nodal cycle.
We set $v_0=(r,\xi,a_0)$.
As in the proof of Lemma \ref{lem:l^2}, we have the following.
\begin{lem}\label{lem:l^2-2}
If ${\cal M}_H(v)^{\mu s} \ne \emptyset$, then
$\langle v^2 \rangle \geq 2l^2$ or
$v=v_0$.
\end{lem}

\begin{prop}
Assume that $\langle v^2 \rangle \geq 2l^2$.
Then ${\cal M}_H(v)^{\mu s} \ne \emptyset$.
Moreover each connected component contains a locally free sheaf.
\end{prop}

\begin{proof}
By Lemma \ref{lem:stack-dim}, we get that
\begin{equation}
\dim {\cal M}_H(v)^{\mu ss} \geq \langle v^2 \rangle. 
\end{equation}
We shall show that 
\begin{equation}
\dim {\cal M}_H(v)^{p \mu ss} \leq \langle v^2 \rangle -
(\langle v^2 \rangle/2l-l+1).
\end{equation}
For this purpose, we shall estimate the moduli number of
Jordan-H\"{o}lder filtrations. 
Let $E$ be a $\mu$-semi-stable sheaf of
$v(E)=v$ and let
\begin{equation}\label{eq:mu-JHF2}
0 \subset F_1 \subset F_2 \subset \cdots \subset F_t=E
\end{equation}
be a Jordan-H\"{o}lder filtration of $E$ with respect to $\mu$-stability.
We set $E_i:=F_i/F_{i-1}$.
We also set
$$
v(E):=l v_0+a \varrho_X,
v(E_i):=l_i v_0+a_i \varrho_X.
$$
By Lemma \ref{lem:l^2-2},
$\langle v_i^2 \rangle \ne 0$.
Hence 
\begin{equation}
\dim {\cal M}_H(v_i)^{\mu s}=
\begin{cases}
\langle v_i^2 \rangle & \text{ if $v_i \ne v_0$}\\ 
\langle v_i^2 \rangle+1=-1 & \text{ if $v_i=v_0$}.
\end{cases}
\end{equation}
%
Since $\langle v^2 \rangle>0$, there is an integer $i_0$ such that
$v_i \ne v_0$. 
Let $J(v_1,v_2,...,v_t)$ be the substack of ${\cal M}_H(v)^{\mu ss}$
such that $E \in {\cal M}_H(v)^{\mu ss}$ has a filtration \eqref{eq:mu-JHF2}.
By using \cite[Lem. 5.2]{K-Y} successively,
we see that
\begin{equation}
\begin{split}
\dim J(v_1,v_2,...,v_t) \leq 
& \sum_i \dim {\cal M}_H(v_i)^{\mu s}+
\sum_{i < j} (\dim \Ext^1(E_j,E_i)-\dim \Hom(E_j,E_i))\\
 \leq &  -\chi(E,E)+\sum_{i>j}\chi(E_j,E_i)+
\sum_{i<j}\dim \Ext^2(E_j,E_i)+(t-1).
\end{split}
\end{equation}
By the same computation of \cite[Lem. 4.4]{Y:K3},
we get the desired estimate.
Hence the existence of a locally free sheaf follows 
by Lemma \ref{lem:l^2-2} and the last paragraph of the proof of
Proposition \ref{prop:stable}.
\end{proof}

By the $(-2)$-reflection associated to $v_0$,
we also get the following.
\begin{prop}
Assume that
$br-l>0$, i.e.,
$\langle(lv_0-b \varrho_X)^2 \rangle=2l(br-l)>0$.
Then  
$$
{\cal M}_H(lv_0-b \varrho_X)^{ss} \cong
{\cal M}_H((br-l)v_0^{\vee}-b \varrho_X)^{ss}.
$$
\end{prop}

\begin{NB}
\begin{cor}
Let $v=(lr,l\xi,\frac{s}{2})$ be a primitive Mukai
vector such that $\gcd(r,\xi)=1$ and $\langle v^2 \rangle \geq 0$.
For $L \in \Pic(X)$ with $c_1(L)=l \xi$,
${\cal M}_H(lr,L,\frac{s}{2})$ contains a $\mu$-stable
sheaf if and only if 
\begin{enumerate}
\item[(i)] 
There is a stable sheaf $E$ such that 
$v(E)=(r,\xi,b)$ and  
$\langle v(E)^2 \rangle=-1$, and
$\langle v^2 \rangle \geq l^2$ or
\item[(ii)] 
There is a stable sheaf $E$ such that 
$v(E)=(r,\xi,b)$ and $\langle v(E)^2 \rangle=-2$, and
$\langle v^2 \rangle \geq 2l^2$ or
\item[(iii)] 
There is no stable sheaf $E$ such that 
$v(E)=(r,\xi,b)$ and 
$\langle v(E)^2 \rangle=-1,-2$
and
$\langle v^2 \rangle \geq 0$.
\end{enumerate}
\end{cor}
\end{NB}

\section{Moduli spaces on elliptic surfaces}\label{sect:elliptic}

\subsection{Some estimates on substacks.}
In this section, we shall study moduli spaces of semi-stable sheaves
on elliptic surfaces. Then we shall apply the results 
to the moduli spaces on Enriques surfaces,
since Enriques surfaces have elliptic fibrations.
 
Let $\pi:X \to C$ be an elliptic surface such that
every fiber is irreducible.
Let $f$ be a fiber of $\pi$.
We have a homomorphism
\begin{equation}
\begin{matrix}
\tau:& K(X)& \to& {\Bbb Z} \oplus \NS(X) \oplus {\Bbb Z}\\
& E & \mapsto & (\rk E,c_1(E),\chi(E)).
\end{matrix}
\end{equation} 
We set $K(X)_{\topo}:=K(X)/\ker \tau$.
For ${\bf e} \in K(X)_{\topo}$,
let ${\cal M}({\bf e})$ be the moduli stack of coherent sheaves $E$
whose topological invariants are ${\bf e}$.
Let ${\cal M}_H({\bf e})^{ss}$ (resp. ${\cal M}_H({\bf e})^{s}$)
be the substack of
${\cal M}({\bf e})$ consisting of 
semi-stable sheaves (resp. stable sheaves).
Let $E$ be a torsion free sheaf on $X$.
${\bf e} \in K(X)$ denotes the class of $E$ in $K(X)$.
Let $H$ be an ample divisor on $X$ and set 
$H_f:=H+nf$, where $n$ 
is a sufficiently large integer depending on ${\bf e}$.
Let $D$ be a curve on $X$ such that $(D,f)=0$.
For a coherent sheaf $F$ on $D$, 
we set 
\begin{equation}
\begin{split}
\deg(F):=& \chi(F)-\chi({\cal O}_D)=\chi(F),\\ 
\deg_E(F):=& \deg(E^{\vee} \otimes F)=
\rk(E) \deg F-(c_1(E),c_1(F))=\chi(E,F).
\end{split}
\end{equation}
\begin{NB}
\begin{lem}
Assume that there is no multiple fibers. Then
${\cal M}_H(0,rlf,dl)^{ss}$
consists of properly semi-stable sheaves which are $S$-equivalent to
$\oplus_{i=1}^l E_i$, $E_i \in {\cal M}_H(0,rf,d)^{ss}=
{\cal M}_H(0,rf,d)^{s}$.
In particular, $\dim {\cal M}_H(0,rlf,dl)^{ss}=l$.  
\end{lem}
\end{NB}

\begin{defn}
A torsion free sheaf $E$ is $f$-semi-stable if
for all subsheaf $F \ne 0$ of $E$,
$$
\frac{(c_1(F),f)}{\rk F} \leq \frac{(c_1(E),f)}{\rk E}.
$$
If the inequality is strict for all subsheaf $F$ of $E$ with
$0<\rk F<\rk E$, then
$E$ is $f$-stable.
$f$-semi-stability of $E$ is equivalent 
to the semi-stability of the restriction $E \otimes k(\eta)$ of $E$ to 
the generic fiber.
Let
${\cal M}_f({\bf e})^{ss}$ 
be the stack of $f$-semi-stable sheaves $E$
with $\tau(E)={\bf e}$.
\end{defn}

\begin{rem}
${\cal M}_f({\bf e})^{ss}$ is not bounded in general.
\begin{NB}
For $E_1 \oplus E_2 \in {\cal M}_f({\bf e})^{ss}$,
$E_1(r_2 n f) \oplus E_2(-r_1 nf) \in {\cal M}_f({\bf e})^{ss}$
for all $n$, where $\rk E_i=r_i$.
\end{NB}
For a positive number $B$, 
let ${\cal M}_f({\bf e})_B^{ss}$
be the open substack of
 ${\cal M}_f({\bf e})^{ss}$ consisting of $E$
such that
for any subsheaf $F$ of $E$,
$$
\frac{(c_1(F),H)}{\rk F} \leq \frac{(c_1(E),H)}{\rk E}+B.
$$
Then ${\cal M}_f({\bf e})_B^{ss}$ is bounded 
(\cite{Ma:open}, \cite{Ma:bdd})
and
${\cal M}_f({\bf e})^{ss}=\cup_B {\cal M}_f({\bf e})_B^{ss}$.
\end{rem}

\begin{NB}
We set $\sigma:=H-m f$, where $m:=\frac{(H^2)}{2(H,f)}f$.
Then $(\sigma^2)=0$ and $H=\sigma+m f$.
Assume that $E$ is $f$-semi-stable but not $H+nf$-semi-stable.
Let $E_1$ be a subsheaf of $E$ such that 
$(\rk E c_1(E_1)-\rk E_1 c_1(E),H+nf) > 0$ and 
$(\rk E c_1(E_1)-\rk E_1 c_1(E),f) \leq 0$.
Then $\rk E c_1(E_1)-\rk E_1 c_1(E)=x \sigma+y f+D$ with
$x \leq 0$, $y \geq -\frac{(n+m)}{(H,f)}x$ and
$D \in \sigma^\perp \cap f^\perp$.
Then  $ ((\rk E c_1(E_1)-\rk E_1 c_1(E))^2)=2xy(H,f)+(D^2)
\leq -2(n+m)x^2+(D^2)$.
If there is a wall between $H+nf$ and $f$, then 
we can bound $- ((\rk E c_1(E_1)-\rk E_1 c_1(E))^2)$ by
$\Delta(E)$ (we use Bogomolov inequality). 
In this case, we may assume that $x<0$. Then
the choice of $x,D$ are bounded.
Moreover if $n \gg \Delta(E)$, then there is no choice of $x$.
Thus there is no wall.
Even if $x=0$ (which corresponds to 
$E \in {\cal M}_f({\bf e})^{ss} \setminus {\cal M}_{H_f}({\bf e})^{ss}$.
In this case, if 
$(\rk E c_1(E_1)-\rk E_1 c_1(E),H)$ is bounded, then we can also see that
the choice of $v(E_1)$ is bounded. 
\end{NB}

%

\begin{defn}
For a coherent sheaf $E$, we set
$$
\Delta(E):=2 \rk E c_2(E)-(\rk E-1)(c_1(E)^2).
$$
\end{defn}

\begin{lem}[Bogomolov inequality]\label{lem:Bogomolov}
If ${\cal M}_f({\bf e})^{ss} \ne \emptyset$, then 
$\Delta({\bf e})\geq 0$.
\end{lem}

\begin{proof}
If $E$ is a $f$-stable sheaf $E$, then
it is $H_f$-stable, and hence
$\Delta(E) \geq 0$ by the Bogomolov inequality.

We next treat the general case.
We note that
$E \in {\cal M}_f({\bf e})^{ss}$
is a succesive extension of 
$f$-stable sheaves $E_i$ with 
$(c_1(E_i),f)/\rk E_i=(c_1(E),f)/\rk E$.
For an extension 
$$
0 \to E_1 \to E \to E_2 \to 0
$$
of $E_i \in {\cal M}_f({\bf e}_i)^{ss}$ with
$(c_1(E_1),f)/\rk E_1=(c_1(E_2),f)/\rk E_2$,
we have
\begin{equation}
\begin{split}
\Delta(E)=&
\rk E \frac{\Delta(E_1)}{\rk E_1}+\rk E \frac{\Delta(E_2)}{\rk E_2}-
\frac{((\rk E_1 c_1(E_2)-\rk E_2 c_1(E_1))^2)}{\rk E_1 \rk E_2}\\
\geq & \rk E \frac{\Delta(E_1)}{\rk E_1}+\rk E \frac{\Delta(E_2)}{\rk E_2}.
\end{split}
\end{equation}
Hence by the induction of $\rk E$, we get the claim.
\end{proof}

\begin{lem}\label{lem:fiber-stable}
For $E \in {\cal M}_f({\bf e})^{ss}$,
there is an exact sequence
\begin{equation}\label{eq:fiber}
0 \to \widetilde{E} \to E \to F \to 0
\end{equation}
such that 
\begin{enumerate}
\item[(i)]
$\widetilde{E}_{|D}$ is a stable purely 1 dimensional 
sheaf for every fiber $D$ with reduced structure, 
\item[(ii)]
$F$ is a purely 1 dimensional sheaf supported on
fibers and 
\item[(iii)] 
$\Hom(E',F)=0$ if $E'$ is a coherent sheaf of rank $r$ on $X$
such that $E'_{|D}$ is a semi-stable sheaf of degree $(c_1(E),D)$ 
for every $D$.
\end{enumerate}
By these properties, $\widetilde{E}$ and $F$ are 
uniquely determined by $E$.   
\end{lem}

\begin{proof}
If $E_{|D}$ is not purely 1-dimensional or
purely 1-dimensional but not semi-stable, then
we take a surjective homomorphism 
$\phi:E \to E_{|D} \to G$ such that
$G$ is a semi-stable 1-dimensional sheaf with 
$\deg_E(G)<0$. 
We set $E':=\ker \phi$.
Then $E'$ is a $f$-semi-stable sheaf with
$\rk E'=\rk E$ and $(c_1(E'),D)=(c_1(E),D)$.
If $E'_{|D}$ is not semi-stable, then
we continue the same procedure.    
Since 
$$
0 \leq \Delta(E')=\Delta(E)+2\deg_E G<\Delta(E),
$$
we finally get a desired subsheaf $\widetilde{E}$ of $E$.
We set $F:=E/\widetilde{E}$.
Since $F$ is a successive extension
of semi-stable 1-dimensional sheaves $G$
with $\deg_E(G)<0$,
we have $\Hom(E',F)=0$.
\end{proof}

For the quotient $F$ of $E$ in \eqref{eq:fiber},
let 
\begin{equation}\label{eq:HN}
0 \subset F_1 \subset F_2 \subset \cdots \subset F_s=F
\end{equation}
be the Harder-Narasimhan filtration
of $F$ with respect to $H$.
We remark that semi-stability is independent of the choice of $H$
by the irreducibility of fibers of $\pi$.
Then 
\begin{equation}\label{eq:F_i}
\Hom(F_i/F_{i-1},F_j/F_{j-1}(K_X))=0,\;i<j.
\end{equation}
By the construction of $F$,
$\deg_E(F_i/F_{i-1})<0$ for all $i$.
In particular,
\begin{equation}\label{eq:{E}}
\Hom(\widetilde{E},F_i/F_{i-1}(K_X))=0
\end{equation}
for all $i$.
Let 
\begin{equation}
F_{i-1}=F_{i-1,0} \subset F_{i-1,1} \subset \cdots 
\subset F_{i-1,n(i-1)}=F_i=F_{i,0}
\end{equation}
be a filtration of $F_i$ such that $F_{i,j}/F_{i,j-1}$ are stable sheaves,
thus $F_i/F_{i-1}$ is $S$-equivalent to
$\oplus_{j=1}^{n(i-1)}F_{i-1,j}/F_{i-1,j-1}$.
We set $E_{i,j}:=\ker(E \to F/F_{i,j})$.
Then we have a filtration
$$
0 \subset \widetilde{E}=E_{0,0} \subset E_{0,1} \subset 
\cdots \subset E_{0,n(0)}=E_{1,0} \subset E_{1,1} \subset \cdots \subset
E_{s-1,n(s-1)}=E_s=E 
$$
such that $E_{i,j}/E_{i,j-1} \cong F_{i,j}/F_{i,j-1}$.
By Lemma \ref{lem:relative-moduli},
$F_{i,j}/F_{i,j-1}$ are stable sheaves on a 
reduced and irreducible divisor $D_{ij}$.
Since $E_{i,j}$ are torsion free and
$E_{i,j} \to F_{i,j}/F_{i,j-1}$ are surjective, we have
$$
\rk E=\rk E_{i,j} \geq \rk F_{i,j}/F_{i,j-1}.
$$
Let ${\bf f}_i \in K(X)_{\topo}$ be the class of $F_i/F_{i-1}$
and $\widetilde{\bf e} \in K(X)_{\topo}$
the class of $\widetilde{E}$.
We set 
$$
(r_{ij} D_{ij},d_{ij}):=
(c_1(F_{ij}/F_{i,j-1}),\chi(F_{ij}/F_{i,j-1})),
$$
 where 
$r_{ij},d_{ij} \in {\Bbb Z}$ and
$\gcd(r_{ij},d_{ij})=1$.
Then 
\begin{equation}
\begin{split}
 0<& r_{ij} \leq   r,\\
-\deg_E(F_{i,j}/F_{i,j-1})= & r_{ij} (c_1(E),D_{ij})-r d_{ij}>0,\\
\Delta(E) = & 2\sum_{i,j} (r_{ij} (c_1(E),D_{ij})-r d_{ij})
+\Delta(\widetilde{E}).
\end{split}
\end{equation}
\begin{NB}
$$
\frac{\chi_E(F_{i,j}/F_{i,j-1})}{(c_1(F_{i,j}/F_{i,j-1}),H)}=
-\frac{r_{ij}(c_1(E),D_{ij})-rd_{ij}}{r_{ij}(D_{ij},H)}=
-\frac{(D_{ij},c_1(E))}{(D_{ij},H)}+\frac{r d_{ij}}{r_{ij}(D_{ij},H)}
$$
and the first term is independent of the choice of$i,j$.
\end{NB}
Hence we see that the choice of ${\bf f}_i$ is finite.
\begin{NB}
$m_{ij} D_{ij}$ is the class of a fiber,
where $m_{ij}$ is the multiplicity.
In particular, $\deg(E_{|\pi^{-1}(t)})/m \leq 
(D_{ij},c_1(E)) \leq 
\deg(E_{|\pi^{-1}(t)})$, where $\pi$ is smooth over
$t$ and $m$ is the maximum of the multiplicity. 
\end{NB}

\begin{prop}\label{prop:codim-general}
Let ${\cal F}(\widetilde{\bf e},{\bf f}_1,...,{\bf f}_s)$
be the stack of filtrations
\begin{equation}
0 \subset \widetilde{E}=\widetilde{F}_0
 \subset \widetilde{F}_1 \subset \widetilde{F}_2
\subset \cdots \subset \widetilde{F}_{s-1} \subset 
\widetilde{F}_s=E
\end{equation}
such that
$\widetilde{E} \in {\cal M}_f(\widetilde{\bf e})^{ss}$ satisfies
(i) in Lemma \ref{lem:fiber-stable} and
$\widetilde{F}_i/\widetilde{F}_{i-1} \in {\cal M}_H({\bf f}_i)^{ss}$
for all $i$.
Then 
\begin{equation}
\dim {\cal F}(\widetilde{\bf e},{\bf f}_1,...,{\bf f}_s)
=-\sum_i \chi({\bf f}_i,\widetilde{\bf e})+
\dim {\cal M}_f(\widetilde{\bf e})^{ss}+
\sum_i {\cal M}_H({\bf f}_i)^{ss}.
\end{equation}
\end{prop}

\begin{proof}
By \eqref{eq:F_i}, \eqref{eq:{E}} and the Serre duality,
we have
\begin{equation}
\begin{split}
\Ext^2(F_j/F_{j-1},F_i/F_{i-1})=0,\;i<j\\
\Ext^2(F_i/F_{i-1},\widetilde{E})=0,\;1 \leq i \leq s.
\end{split}
\end{equation}
Then the proof of \cite[Lem. 5.3]{K-Y} implies that
\begin{equation}
\begin{split}
\dim {\cal F}(\widetilde{\bf e},{\bf f}_1,...,{\bf f}_s)
=& -\sum_i \chi({\bf f}_i,\widetilde{\bf e})-
\sum_{i<j}\chi({\bf f}_j,{\bf f}_i)+
\dim {\cal M}_f(\widetilde{\bf e})^{ss}+
\sum_i {\cal M}_H({\bf f}_i)^{ss}\\
=& -\sum_i \chi({\bf f}_i,\widetilde{\bf e})+
\dim {\cal M}_f(\widetilde{\bf e})^{ss}+
\sum_i {\cal M}_H({\bf f}_i)^{ss}.\\
\end{split}
\end{equation}
\end{proof}

\begin{lem}\label{lem:relative-moduli}
Let $D$ be a reduced and irreducible curve on $X$ with $(D^2)=0$.
For an element $G_1 \in K(X)$ with
$\rk G_1>0$,
let $E$ be a $G_1$-twisted stable purely 1-dimensional sheaf such that
$\Div(E)=rD$ and $\chi(E)=d$.
Then $E$ is a stable sheaf on $D$.
In particular $\gcd(r,d)=1$. 
\end{lem}

\begin{proof}
We note that ${\cal O}_D(D)$ is a numerically trivial
line bundle on $D$.
Let $T$ be the torsion submodule of $E_{|D}$.
Then $E':=E_{|D}/T$ has the Harder-Narasimhan filtration
\begin{equation*}
0=F_0 \subset F_1 \subset F_2 \subset \cdots \subset F_s=E'.
\end{equation*}
We set $(c_1(F_i/F_{i-1}),\chi(F_i/F_{i-1})):=(r_i D,d_i)$.
Then $r_i \in {\Bbb Z}_{>0}$, $\sum_{i=1}^s r_i \leq r$ and
$$
\frac{d_1}{r_1}>\frac{d_2}{r_2}>\cdots > \frac{d_s}{r_s}.
$$
By the $G_1$-twisted stability of $E$,
we have
\begin{equation*}
0=\frac{(\rk G_1) d-(c_1(G_1),rD)}{(rD,H)} \leq 
\frac{(\rk G_1) d_s-(c_1(G_1),r_s D)}{(r_s D,H)}.
\end{equation*}
Hence $d/r \leq d_s/r_s$.

There is a positive integer $k$ such that
$E$ is an ${\cal O}_{(k+1)D}$-module and
$E(-kD) \overset{kD}{\to} E$ is non-zero.
Then we have a non-zero homomorphism
$E_{|D}(-kD) \to E$.
Then $\Hom(F_i/F_{i-1}(-kD),E) \ne 0$ for some $i$,
which implies $d_i/r_i \leq d/r$.
Then $i=s$ and $d/r=d_s/r_s$.
By the stability of $E$,
$E \to F_s/F_{s-1}$ is an isomorphism.
In particular, $E$ is a stable sheaf on $D$.
Since $D$ is a reduced and irreducible curve of $g(D)=1$,  
there is an elliptic surface $X'$ with a section 
such that $D$ is a fiber.
We set $m:=\gcd(r,d)$ and $(r',d'):=(r/m,d/m)$.
For a general polarization $H'$ on $X'$,
we consider the moduli space $Y:=M_{H'}(0,r'f,d')$ of 
stable sheaves $F$ of dimension 1 on $X$
whose Chern character is $(0,r' f,d')$, where
$f$ is a fiber. Let ${\cal E}$ be a universal family.
Then by a general theory of Fourier-Mukai transform,
we see that $E$ is a successeive extension of 
${\cal E}_{|X' \times \{ y \}}$ $(y \in Y)$.
Since $E$ is stable, $m=1$. Therefore $\gcd(r,d)=1$. 
\end{proof}

\begin{NB}
The case where $(c_1(E),\chi_{G_1}(E))=(rD,d)$ and $D$ is reducible 
affine A,D,E configuration.

\begin{lem}
Let $mD$ be a fiber of an elliptic fibration, where
$m$ is the multiplicity.
Let $E$ be a $G_1$-twisted stable purely 1-dimensional sheaf with 
$(c_1(E),\chi_G(E))=(rD,d)$ whose support is $D$.
Then $E$ is a $G_1$-twisted stable sheaf on $D$. 
\end{lem}

\begin{proof}
We note that ${\cal O}_D(D)$ is a numerically trivial
line bundle on $D$.
Let $T$ be the torsion submodule of $E_{|D}$.
Then $E':=E_{|D}/T$ has the Harder-Narasimhan filtration
\begin{equation*}
0=F_0 \subset F_1 \subset F_2 \subset \cdots \subset F_s=E'.
\end{equation*}
We set $(c_1(F_i/F_{i-1}),\chi_{G_1}(F_i/F_{i-1})):=(\xi_i,d_i)$.
Then $(\xi_i,H) \in {\Bbb Z}_{>0}$, $\sum_{i=1}^s (\xi_i,H) \leq (rD,H)$ and
$$
\frac{d_1}{(\xi_1,H)}>\frac{d_2}{(\xi_2,H)}>\cdots > \frac{d_s}{(\xi_s,H)}.
$$
By the $G_1$-twisted stability of $E$,
we have
\begin{equation*}
\frac{d}{(rD,H)} \leq 
\frac{d_s}{(\xi_s,H)}.
\end{equation*}

There is a positive integer $k$ such that
$E$ is an ${\cal O}_{(k+1)D}$-module and
$E(-kD) \overset{kD}{\to} E$ is non-zero.
Then we have a non-zero homomorphism
$E_{|D}(-kD) \to E$.
Then $\Hom(F_i/F_{i-1}(-kD),E) \ne 0$ for some $i$,
which implies $d_i/(\xi_i,H) \leq d/r$.
Then $i=s$ and $d/(rD,H)=d_s/(\xi_s,H)$.
By the stability of $E$,
$E \to F_s/F_{s-1}$ is an isomorphism.
In particular, $E$ is a stable sheaf on $D$. 
\end{proof}
\end{NB}

\begin{cor}
For $E \in {\cal M}_H(0,rf,d)^{ss}$, we have a decomposition
$E \cong \oplus_i E_i$ such that
$D_i :=\Supp(E_i)$ are fibers of $\pi$ with reduced scheme structure,  
$E_i$ are successive extension of stable sheaves $E_{ij}$ on $D_i$
with $\frac{\chi(E_{ij})}{\rk(E_{ij})(D_i,H)}=\frac{d}{r(f,H)}$,
and $D_i \cap D_j=\emptyset$ for $i \ne j$.
\end{cor}

\begin{lem}\label{lem:nD-stability}
Let $D$ be a reduced and irreducible curve on $X$ with $(D^2)=0$.
For a torsion free sheaf $E$ on $X$,
$E_{|nD}$ is semi-stable if and only if 
$E_{|D}$ is semi-stable.
Moreover if $E_{|D}$ is semi-stable, then $E$ is locally free
in a neighborhood of $D$.
\end{lem}

\begin{proof}
We have a filtration
\begin{equation}
0 \subset E(-nD) \subset E(-(n-1)D) \subset E(-(n-2)D) \subset
\cdots \subset E(-D) \subset E.
\end{equation}
We set $L:={\cal O}_D(-D)$. Then 
$E(-kD)/E(-(k+1)D) \cong E_{|D} \otimes L^{\otimes k}$
and $\chi(E_{|D} \otimes L^{\otimes k})=\chi(E_{|D})$ for 
$0<k<(n-1)$.
Hence $E_{|nD}$ is semi-stable if and only if $E_{|D}$ 
is semi-stable. 
If $E_{|D}$ 
is semi-stable, then $E_{|D}$ is purely 1-dimensional,
which shows that $E$ is locally free in a neighborhood of $D$.
\end{proof}

\subsection{For the case of an unnodal Enriques surface}
\label{subsect:unnodal}

Let $X$ be an unnodal Enriques surface.
Let $U:={\Bbb Z}e_1+{\Bbb Z}e_2$ be a hyperbolic sublattice of the lattice
$H^2(X,{\Bbb Z})_f$.
We assume that $e_1,e_2$ are effective and $|2e_1|$ gives an elliptic fibration
$\pi:X \to {\Bbb P}^1$. 
Let $2\Pi_1,2\Pi_2$ be the multiple fibers of $\pi$.
 Let $\eta \in {\Bbb P}^1$ be the generic point of
${\Bbb P}^1$.
Let 
$u:=(r,d e_2,0)$ be a primitive and isotropic Mukai vector.
We note that $r$ is even.
We assume that $(r,d)=1$. \begin{NB} This is a strong assumption.
Indeed $(2,2e_2,0)$ is a primitive and isotropic Mukai vector. \end{NB}
For a Mukai vector $v \in (0,re_1,d)^\perp$,
we can write $v=lu+n e_1+\delta+a \varrho_X$
where $l,n,a \in {\Bbb Z}$ and $\delta \in U^\perp$.
If $v$ is primitive and $\ell(v)=2$, then 
$2 \mid l, 2 \mid n, 2 \mid \delta$ and $2 \nmid a$.
\begin{NB}
Note that $d$ is odd.
$\chi(v)=lr/2+a$ is odd.
\end{NB}
We can easily show the following claims.
\begin{lem}\label{lem:f}
Let $v_i:=l_i u+n_i e_1+\delta_i+a_i \varrho_X$ $(i=1,2)$
be two Mukai vectors with $l_i,n_i,a_i \in {\Bbb Z}$
and $\delta_i \in U^\perp$.
\begin{enumerate}
\item[(1)]
\begin{equation}
\langle v_1,v_2 \rangle
=\frac{l_2}{2l_1}\langle v_1^2 \rangle+\frac{l_1}{2l_2}\langle v_2^2 \rangle
-\frac{1}{2l_1 l_2}\langle (l_2 \delta_1-l_1 \delta_2)^2 \rangle.
\end{equation}
\item[(2)]
If $\ell(v_1)=2$, then
\begin{equation}
\langle v_1,v_2 \rangle=
(l_1 n_2+l_2 n_1)d+(\delta_1,\delta_2)-(l_1 a_2+l_2 a_1)r
\in 2{\Bbb Z}.
\end{equation}
Moreover if $\ell(v_2)=2$ also holds, then
$\langle v_1,v_2 \rangle \in 4{\Bbb Z}$.
\end{enumerate}
\end{lem}


\begin{NB}
It is a union of substacks such that
$\{E \in {\cal M}(v) \mid \text{ $E_{|f}$ is semi-stable on a fiber $f$}\}$.
\end{NB}

Let $E$ be a $f$-stable sheaf with $v(E)=lu+n e_1+\delta+a \varrho_X$,
where $l,n,a \in {\Bbb Z}$ and $\delta \in U^\perp$.
Since the $f$-stability implies the $H_f$-stability, 
$\rk E=lr$ is even and $X$ is unnodal, we have
$\langle v(E)^2 \rangle \geq 0$. 
Then as in the proof of Lemma \ref{lem:Bogomolov}, by using Lemma \ref{lem:f},
we get the following inequality.
\begin{lem}\label{lem:Bogomolov2}
If ${\cal M}_f(v)^{ss} \ne \emptyset$, then
$\langle v^2 \rangle \geq 0$.
\end{lem}

\begin{prop}\label{prop:horizontal}
Assume that $r$ is even and $(r,d)=1$.
Assume that $\langle v^2 \rangle>0$. Then
${\cal M}_{H_f}(v)^{ss}$ is an open and dense substack of
${\cal M}_f(v)^{ss}$.
In particular, $\dim {\cal M}_f(v)^{ss}= \langle v^2 \rangle$.
\end{prop}

\begin{proof}
We set $H:=H_f$. 
It is sufficient to prove the claim
for bounded substacks ${\cal M}_f(v)_B^{ss}$,
$B \in {\Bbb Q}$.
For $F \in {\cal M}_f(v)^{ss}$,
let 
\begin{equation}
 0 \subset F_1 \subset F_2 \subset \dots \subset F_s=F
\end{equation}
be the Harder-Narasimhan filtration of $F$.
Then the choice of $v(F_i/F_{i-1})$ $(1 \leq i \leq s)$
is finite. 
Replacing $H$ by $H'$ in a neighborhood of $H$,
we may assume that $H$ is a general polarization
with respect to $v(F_i/F_{i-1})$ $(1 \leq i \leq s)$.
We set 
\begin{equation}\label{eq:v-fiber}
 v_i:=v(F_i/F_{i-1})=l_i u+n_i e_1+\delta_i+a_i \varrho_X, 1 \leq i \leq s.
\end{equation}
By Lemma \ref{lem:Bogomolov2}, $\langle v_i^2 \rangle \geq 0$
for all $i$.
Let ${\cal F}^{HN}(v_1,v_2,\dots,v_s)$ be the substack of 
${\cal M}_H(v)^{\mu ss}$
whose element $E$ has the Harder-Narasimhan filtration
of the above type.
We shall prove that
\begin{equation}\label{eq:HN-hor}
\dim {\cal F}^{HN}(v_1,v_2,\dots,v_s) < \langle v^2 \rangle.
\end{equation}
Since $\Hom(F_i/F_{i-1},F_j/F_{j-1}(K_X))=0$ for $i<j$,
\cite[Lem. 5.3]{K-Y} implies that
\begin{equation}
 \begin{split}
  \dim {\cal F}^{HN}(v_1,v_2,\dots,v_s)=&
  \sum_{i=1}^s \dim {\cal M}_H(v_i)^{ss}+
  \sum_{i<j}\langle v_j,v_i \rangle\\
  =& \langle v^2 \rangle-\left(\sum_{i<j}\langle v_j,v_i \rangle+
 \sum_{i=1}^s (\langle v_i^2 \rangle-\dim {\cal M}_H(v_i)^{ss})\right).
 \end{split}
\end{equation}
If $v_i$ is isotropic, then 
we write $v_i=k_i u_i$, where $u_i$ is primitive and 
$k_i \in {\Bbb Z}_{>0}$.
Since $\gcd(r,d)=1$, $\dim {\cal M}_H(u_i)^{ss}=0,1$
according as $\ell(u_i)=1,2$.
\begin{NB}
Old version, which is wrong!!
$\ell(u_i)=1$.
It may happen that $\ell(u_i)=2$.
Thus we may have $\dim {\cal M}_H(u_i)^{ss}=1$.
\end{NB}
By Proposition \ref{prop:isotropic},
$\dim {\cal M}_H(v_i)^{ss} \leq k_i/2,k_i$.
If there is $v_p$ with $\langle v_p^2 \rangle>0$,
then Lemma \ref{lem:f} implies
$\langle v_p,v_i \rangle=k_i \langle v_p,u_i \rangle \geq k_i,2k_i$
according as $\ell(u_i)=1,2$.
Hence \eqref{eq:HN-hor} holds.
Assume that all $v_i$ are isotropic.
By Lemma \ref{lem:f},
$\langle v,v_i \rangle>0$ for all $i$.
Hence for all $i$, there is a positive integer $n(i)$ such that
$\langle v_i, v_{n(i)} \rangle>0$.
We set $\epsilon:=\ell(u_i)+\ell(u_j)-2$.
Then 
$$
2^{\epsilon} k_i k_j-\dim{\cal M}_H(v_i)^{ss}-\dim{\cal M}_H(v_j)^{ss}
>0.
$$
\begin{NB}
$k_i k_j-k_i/2-k_j/2=k_i(k_j-1)/2+k_j(k_i-1)/2
\geq 0$ and the equality holds only if $k_i=k_j=1$.
In this case, $k_i k_j-[k_i/2]-[k_j/2]=1$.
We also have $2k_i k_j-k_i-k_j/2>0$ and
 $4k_i k_j-k_i-k_j>0$.
\end{NB}
Hence
$$
\sum_{i \in \{i \mid i<n(i)\}} \langle v_i,v_{n(i)} \rangle
>\sum_i \dim {\cal M}_H(v_i)^{ss}.
$$
Therefore \eqref{eq:HN-hor} holds.
\end{proof}

\begin{lem}\label{lem:horizontal-isotropic}
If $v=lu+n e_1+\delta+a \varrho_X$ is isotropic, then
we have $\dim {\cal M}_f(v)^{ss} \leq \left[\frac{l}{2}\right]$.
\end{lem}

\begin{proof}
Any $v_i$ in \eqref{eq:v-fiber} is isotropic and 
$\dim {\cal M}_H(v_i)^{ss}=[k_i/2],k_i$ accordng as
$\ell(u_i)=1,2$.
If $\ell(v_i)=2$, then $l_i \geq 2k_i$.
Hence $\dim {\cal M}_H(v_i)^{ss} \leq l_i/2$ for all $i$.
Therefore 
$$
\dim {\cal F}^{HN}(v_1,v_2,...,v_s)=
\sum_i \dim {\cal M}_H(v_i)^{ss} \leq \frac{l}{2}.
$$
\end{proof}

\begin{defn}\label{defn:flat}
\begin{enumerate}
\item[(1)]
Let 
${\cal M}_H(v)_*^{ss}$ be the open substack of
${\cal M}_H(v)^{ss}$ consisting of $E$ such that
$E_{|\pi^{-1}(t)}$ is semi-stable for all $t \in {\Bbb P}^1$.
By Lemma \ref{lem:nD-stability},
${\cal M}_H(v)^{ss}$ consisting of $E$ such that
$E_{|\pi^{-1}(t)_{\mathrm{red}}}$ 
is semi-stable for all $t \in {\Bbb P}^1$,
where $\pi^{-1}(t)_{\mathrm{red}}$ is the reduced part of
$\pi^{-1}(t)$.
\item[(2)]
Let 
${\cal M}_f(v)_*^{ss}$ be the open substack of
${\cal M}_f(v)^{ss}$ consisting of $E$ such that
$E_{|\pi^{-1}(t)_{\mathrm{red}}}$ 
is semi-stable for all $t \in {\Bbb P}^1$.
\end{enumerate}
\end{defn}

\begin{prop}\label{prop:flat}
We set $v:=lu+n e_1+\delta+a \varrho_X$, where
$l,n,a \in {\Bbb Z}$, $l>0$ and $\delta \in U^\perp$.
Then ${\cal M}_f(v)_*^{ss}$ is an open and dense substack
of ${\cal M}_f(v)^{ss}$.
\end{prop}

\begin{proof}
For $E \in {\cal M}_f(v)^{ss}$, we have the filtration \eqref{eq:fiber}.
For the filtration \eqref{eq:HN},
we set 
$$
v_i:=v(F_i/F_{i-1})=k_i (0,r_i e_1,d_i), 
$$
where 
$(0,r_i e_1,d_i)$ are primitive.
By Proposition \ref{prop:codim-general},
\begin{equation}
\langle v^2 \rangle-\dim {\cal F}(\tilde{v},v_1,...,v_s)=
\sum_i lk_i(r_i d-r d_i)-
\sum_i \dim {\cal M}_H(v_i)^{ss}-(\dim {\cal M}_f(\tilde{v})^{ss}-
\langle \tilde{v}^2 \rangle).
\end{equation}
We first assume that $\tilde{v}$ is isotropic.
Then $\dim {\cal M}_f(\tilde{v})^{ss}-\langle \tilde{v}^2 \rangle
\leq [\tfrac{l}{2}]$ by Lemma \ref{lem:horizontal-isotropic}.
If $\ell(v_i)=2$, then 
$r_i$ is even.
In this case, we have $(r_i d-r d_i) \in 2{\Bbb Z}$.
Hence
\begin{equation}
\begin{split}
& lk_i(r_i d-r d_i)-
\dim {\cal M}_H(v_i)^{ss}-(\dim {\cal M}_f(\tilde{v})^{ss}-
\langle \tilde{v}^2 \rangle)\\
\geq & \min \left \{lk_i-\left[\frac{k_i}{2}\right]-\left[\frac{l}{2}\right],
2lk_i-k_i-\left[\frac{l}{2} \right] \right\}>0.
\end{split}
\end{equation}
If $\tilde{v}$ is not isotropic, then
by using Proposition \ref{prop:horizontal},
we get $\langle v^2 \rangle-\dim {\cal F}(\tilde{v},v_1,...,v_s)>0$.
Therefore our claim holds.
\end{proof}

\begin{NB}
Assume that $l=1$ and $v=(2,\xi,a)$ with $(\xi,e_1)=1$.
If $k_1=1$ and $s=1$, 
then for $v_1=(0,e_1,0)$ and $v_1=(0,4e_1,1)$,
$\langle v^2 \rangle-\dim {\cal F}(\tilde{v},v_1,...,v_s)=1$.
If $k_1=2$, then
for $2v_1=(0,2e_1,0)$,
$\langle v^2 \rangle-\dim {\cal F}(\tilde{v},v_1,...,v_s)=1$.

Assume that $l=2$.
Then $k_i=1$, $s=1$ and $r_i d-r d_i=1$.

\end{NB}

By the proof of Proposition \ref{prop:flat},
we can compute the boundary components of ${\cal M}_f(v)_*^{ss}$.
Indeed we see that
\begin{equation}
lk_i(r_i d-r d_i)-
\dim {\cal M}_H(v_i)^{ss}-(\dim {\cal M}_f(\tilde{v})^{ss}-
\langle \tilde{v}^2 \rangle)=1
\end{equation}
implies that $(l,k_i)=(1,1),(1,2),(2,1)$.
\begin{NB}
Indeed
if $l \geq 3$, then
$lk_i-[k_i/2]-[l/2] \geq 
lk_i-k_i/2-l/2=l(k_i-1)/2+k_i(l-1)/2 \geq
k_i+3(k_i-1)/2 \geq 2$ for $k_i >1$, and
$lk_i-[k_i/2]-[l/2] \geq l-[l/2] \geq 2$
for $k_i=1$.
We also have 
$2lk_i-k_i-[l/2]\geq 2l-1-[l/2] \geq 2$ for $l \geq 2$.

If $l=2$ and $k_i \geq 2$, then
$lk_i-[k_i/2]-[l/2] \geq l(k_i-1)/2+k_i(l-1)/2 \geq (l+k_i)/2 \geq 2$.
If $l=2$ and $k_i=1$, then
$lk_i-[k_i/2]-[l/2]=1$.

If $l=1$, then
$lk_i-[k_i/2]-[l/2]=k_i-[k_i/2] \geq 2$ for $k_i \geq 3$ and
$lk_i-[k_i/2]-[l/2]=k_i-[k_i/2]=1$ for $k_i=1,2$. 
\end{NB}
If $\langle v^2 \rangle-\dim {\cal F}(\tilde{v},v_1,...,v_s)=1$,
then we see that
$s=1$ and $r_i d-r d_i=1$ for $\ell(v_i)=1$
and $s=1$ and $r_i d-r d_i=2$ for $\ell(v_i)=2$.
Thus a general member $E$ of ${\cal M}_f(v)^{ss} 
\setminus {\cal M}_f(v)_*^{ss}$
fits in an extension
\begin{equation}\label{eq:fiber2}
0 \to E' \to E \to F \to 0
\end{equation}
where $E' \in {\cal M}_f(v_1)^{ss}$ and $F \in {\cal M}_H(v_2)^{ss}$.

Assume that $l=1$.
We take integers $(p,q)$ such that
$0<p \leq r$ and $pd-rq=1$
and set $u_1:=(0,p e_1,q)$.
Let
${\cal F}(v-u_1,u_1)^s$ be the open substack
parameterizing torsion free sheaves $E$ fitting 
in the extension \eqref{eq:fiber2}
such that $\Div(F)=p \Pi_i$. 
Then it defines a divisor on ${\cal M}_f(v)^{ss}$.
Let ${\cal D}_i$ $(i=1,2)$ be the divisor.

We note that 
${\cal M}_H(2u_1,pf)^s$ consists of stable locally free sheaves
of rank $p$ and degree $2q$ on a smooth fiber.
Let ${\cal F}(v-2u_1,2u_1)^s$ be the open substack
parameterizing torsion free sheaves $E$ fitting 
in the extension \eqref{eq:fiber2}
such that $\Div(F)=p f$.
Then it defines a divisor ${\cal D}_3$ on ${\cal M}_f(v)^{ss}$.

We set
$u_2:=(0,2p' e_1,2q')$, where $0<p' \leq r$,
$(p',q')=(p \pm r/2,q \pm d/2)$.
Then $u_2$ is a primitive Mukai vector
with $\langle v,u_2 \rangle=(2p')d-r(2q')=2$.
${\cal M}_H(u_2,p' f)^s$ consists of stable locally free sheaves
of rank $p'$ and degree $2q'$ on a smooth fiber.
Let ${\cal F}(v-u_2,u_2)^s$ be the open substack 
parameterizing torsion free sheaves $E$
fitting in the extension \eqref{eq:fiber2}
such that $\Div(F)=p'f$.
Then it defines a divisor ${\cal D}_4$ 
on ${\cal M}_f(v)^{ss}$.

Since ${\cal M}_f(v)^{ss}={\cal M}_{H_f}(v)^{ss}$,
we have
\begin{equation}\label{eq:decompo}
{\cal M}_{H_f}(v)^{ss} ={\cal M}_{H_f}(v)_*^{ss} \cup
{\cal D}_1 \cup {\cal D}_2 \cup {\cal D}_3 \cup {\cal D}_4
\end{equation}
up to codimension 2.

\begin{ex}\label{ex:u}
For $v=(2,e_2+ne_1+\delta,a)$ with $u=(2,e_2,0)$,
we have $(p,q)=(1,0)$ and $(p',q')=(p,q)+\frac{1}{2}(r,d)=(2,\frac{1}{2})$.
In particular
$u_1=(0,e_1,0)$ and $u_2=(0,4e_1,1)$.
As we shall see in the next section, \eqref{eq:decompo}
holds without removing codimension 2 subset.
\end{ex}

\section{Irreducibility}\label{sect:irreducible}
\subsection{Unnodal case}
Assume that $X$ is an unnodal Enriques surface
and $f$ a smooth fiber of an elliptic fibration
$\pi:X \to {\Bbb P}^1$.
Let $v=(r,\xi,\frac{s}{2})$ be a primitive Mukai vector such that
$r$ is even.
Then for $L \in \NS(X)$ with $[L \mod K_X]=\xi$,
we have an equality of "Hodge polynomials" of the stacks 
defined in \cite{Y:twist1} (see \cite[Prop. 2.4, Thm. 2.6]{Enriques}):
$$
e({\cal M}_H(v,L+\tfrac{r}{2}K_X)^{ss})=e({\cal M}_H(v',L'+K_X)^{ss}),
$$
where $v'=(2,\zeta,\frac{s'}{2})$, $[L' \mod K_X]=\zeta$,
$L \equiv L' \mod 2$, $\langle v^2 \rangle=\langle {v'}^2 \rangle$ and
$\zeta = 0$ if $\ell(v)=2$.
\begin{NB}
\begin{ex}
For $(r_1,d_1)=(2,1)$, if $(\xi,e_1)=1$, then 
$v=(2,\xi,\frac{s}{2})$ and $v=(4,2\xi,\frac{s}{2})$
satisfies the conditions.
$e({\cal M}_H(4,2(e_2+me_1),\frac{s}{2})^{ss})=
e({\cal M}_H(2,0,-2(m-\frac{s}{2}))^{ss})$,
where $s/2$ is odd.
\end{ex}
\end{NB}
Assume that $\ell(v)=\ell(v')=2$.
Then
$v'=(2,0,-2n)$ for some $n \in {\Bbb Z}$.
We set $v''=(4,2(e_2+(n+1)e_1),1)$. Then
$e({\cal M}_H(v',L'+K_X)^{ss})=e({\cal M}_H(v'',L'')^{ss})$,
where $L' \equiv L'' \mod 2$.
In order to prove the irreducibility of
${\cal M}_H(v,L)^{ss}$, it is sufficient to prove the
irredicibility for the following two cases:
\begin{enumerate}
\item[(1)]
$v=(2,e_2+ne_1+\delta,a)$ 
\item[(2)]
 $v=(4,2(e_2+(n+1)e_1),1)$.
\end{enumerate}
In particular, $u=(2,e_2,0)$ in the notation of subsection \ref{subsect:unnodal}. 
We note that $M_H(0,f,1)$ is a fine moduli space
and it is isomorphic to $X$.
$M_H(0,f,1)$ parametrizes torsion free
sheaves of rank 1 on a reduced and irreducible fiber
$\pi^{-1}(t)$ and stable vector bundles of rank 2 and degree 1 on
$\Pi_i$.
Thus
$$
M_H(0,f,1)=\bigcup_{t \in {\Bbb P}^1_0} 
\overline{\Pic}_{\pi^{-1}(t)}^1 \cup M_{\Pi_1}(2,1) \cup M_{\Pi_2}(2,1)
$$
where $\pi$ has reduced fibers over ${\Bbb P}^1_0$,
$\overline{\Pic}_{\pi^{-1}(t)}^1$ are the compactified Jacobian
of degree 1 and 
$M_{\Pi_i}(2,1)$ are the moduli space of stable vector bundles
of rank 2 and degree 1 on $\Pi_i$.  
We take an identification $M_H(0,f,1) \cong X$.  
Let 
${\cal E}$ be a universal family on $X \times X$.
By \cite{Br:1},
\begin{equation}\label{eq:FM1}
{\cal E}_{|X \times \{ x \}} \otimes K_X \cong {\cal E}_{|X \times \{ x \}},
\; x \in X 
\end{equation}
and
\begin{equation}
\begin{matrix}
\Phi_{X \to X}^{{\cal E}}:& {\bf D}(X) & \to & {\bf D}(X)\\
& E & \mapsto & {\bf R}\Hom_{p_2}(p_1^*(E) \otimes {\cal E})
\end{matrix}
\end{equation}
is an equivalence, that is, a Fourier-Mukai transform,
where $p_i:X \times X \to X$  $(i=1,2)$ are projections.
We consider a contravariant Fourier-Mukai transform
\begin{equation}
\begin{matrix}
\Psi:& {\bf D}(X) & \to & {\bf D}(X)\\
& E & \mapsto & {\bf R}\Hom_{p_2}(p_1^*(E),{\cal E}).
\end{matrix}
\end{equation}
Since $\Psi({\cal O}_X)$ is a line bundle, replacing the universal family,
we may assume that
$\Psi({\cal O}_X)={\cal O}_X$.
We set $\Psi^i(E):=H^0(\Psi(E)[i]) \in \Coh(X)$.

\begin{lem}\label{lem:Psi}
$\Psi(0,0,1)=(0,2e_1,1)$, 
$\Psi(0,4e_1,1)=(0,-2e_1,1)$ and $\Psi(0,e_1,0)=(0,-e_1,0)$.
\end{lem}

\begin{proof}
We note that 
$\Psi(k_x)={\cal E}_{|\{x\} \times X}[-2]$.
Hence $c_1(\Psi(k_x))=2e_1$.
Since $1=\chi({\cal O}_X,k_x)=
-\langle \Psi(k_x),\Psi({\cal O}_X) \rangle$,
we have $\Psi(0,0,1)=(0,2e_1,1)$.
Since $\Psi({\cal E}_{|X \times \{x \}})=k_x[-2]$,
$\Psi(0,2e_1,1)=(0,0,1)$.
Since 
\begin{equation}
\begin{split}
(0,4e_1,1)=& 2(0,2e_1,1)-(0,0,1),\\
2(0,e_1,0)=& (0,2e_1,1)-(0,0,1),
\end{split}
\end{equation}
we have 
\begin{equation}
\begin{split}
\Psi(0,4e_1,1)=& (0,-2e_1,1),\\
\Psi(0,e_1,0)=& (0,-e_1,0).
\end{split}
\end{equation}
\end{proof}

For cases (1) and (2),
by \cite[Prop. 3.4.5]{PerverseII}, $\Psi$ induces an isomorphism
${\cal M}_H(v)^{ss} \to {\cal M}_{H'}^{G'}(w)^{ss}$,
where $w=(0,\xi, a)$ with $(\xi,e_1)=1,2$,
${\cal M}_{H'}^{G'}(w)^{ss}$ is the moduli space 
of $G'$-twisted semi-stable sheaves,
and $H' \in \NS(X)_{\Bbb Q}$ and $G' \in K(X)$ 
depend on the choice of $H$ and $v$.
\begin{rem}
Since $H$ is a general polarization, 
${\cal M}^G_H(v)^{ss}$ is independent of the choice of
$G$. Hence we do not need to consider twisted semi-stability. 
\end{rem}
We have a support map
\begin{equation}
\begin{matrix}
\varphi: & {\cal M}_{H'}^{G'}(w,L)^{ss} & \to & |L|\\
& E & \mapsto & \det E.
\end{matrix}
\end{equation}

\begin{lem}\label{lem:FM-flat}
${\cal M}_H(v)_*^{ss}$ is isomorphic to the open substack
${\cal M}_{H'}^{G'}(w)^{ss}_*$ of 
${\cal M}_{H'}^{G'}(w)^{ss}$  
consisting of 
semi-stable sheaves $E$ such that
 $\Div(E)$ is flat over ${\Bbb P}^1$. 
\end{lem}

\begin{proof}
As we remarked, 
$\Psi$ induces an isomorphism
${\cal M}_H(v)^{ss} \to {\cal M}_{H'}^{G'}(w)^{ss}$.

For $E \in {\cal M}_H(v)_*^{ss}$,
the semi-stability of $E_{|\pi^{-1}(t)}$ implies that 
$E_{|\pi^{-1}(t)}$ is a successive extension of 
${\cal E}_{|X \times \{x \}}$ $(x \in \pi^{-1}(t))$.
Hence 
$\Hom(E,{\cal E}_{|X \times \{ x\}})=0$ for a general point of
$x \in \pi^{-1}(t)$.
Hence $\Supp(\Psi^1(E))$ does not contain a fiber, which implies
$\Div(\Psi^1(E))$ is flat over ${\Bbb P}^1$.

Conversely for $L \in {\cal M}_{H'}^{G'}(w)^{ss}_*$,
$L^*:=L^{\vee}[1]$ is a purely 1-dimensional sheaf on $X$.
Hence  
$E:=\Ext_{p_1}^1(p_2^*(L),{\cal E})=p_{1*}(p_2^*(L^*) \otimes {\cal E})$ 
is a locally free sheaf on $X$ such that $E_{|\pi^{-1}(t)}$
is semi-stable for all $t \in {\Bbb P}^1$.
Therefore the claim holds.
\end{proof}

\begin{NB}
For $(4,2D,2n+1)=2(2,D,n)+(0,0,1)$ with $(D,e_1)=1$,
$\Phi(4,2D,2n+1)=2(0,D',n')+(0,2e_1,1)$.
Hence the first Chern class is divisible by 2.
\end{NB}

For case (1), we have
$w=(0,\xi,a)$ with $(\xi,e_1)=1$.
\begin{NB}
$w e^{ke_1}=(0,\xi,a)e^{ke_1}=(0,\xi,a+k)$.
Thus we can take arbitrary $a$.
\end{NB}
We take $L \in \NS(X)$ with $[L \mod K_X]= \xi$.
Then for $E \in {\cal M}_{H'}^{G'}(w,L)^{ss}_*$,
$\Div(E)$ is integral.
Let $|L|_*$ be the open subscheme parametrizing
integral curves.  Then
$\varphi:{\cal M}_{H'}^{G'}(w,L)_*^{ss} \to |L|_*$ is a family of compactified
jacobians over $|L|_*$.
By \cite{AIK}, all fibers are irreducible of dimension
$(L^2)/2$.
Hence ${\cal M}_{H'}^{G'}(w,L)^{ss}_*$ is irreducible.
Thus by Proposition \ref{prop:flat},
we have the following. 
\begin{prop}\label{prop:case(1)}
 ${\cal M}_{H'}^{G'}(w,L)^{ss}$ is irreducible.
\end{prop}

We note that $|L| \setminus |L|_*$ is a divisor consisting of three irreducible
components:
\begin{equation}
\begin{split}
&\{ D \in |L| \mid \Pi_i \subset D \}\;\;(i=1,2),\\
&\{D  \in |L| \mid \pi^{-1}(t) \subset D, t \in {\Bbb P}^1 \}.
\end{split}
\end{equation}
Hence \eqref{eq:decompo} holds without removing codimension 2 subsets. 
Let $u_1$ and $u_2$ be the Mukai vectors in Example \ref{ex:u}.
From the exact sequence
$$
0 \to \widetilde{E} \to E \to F \to 0 
$$
in \eqref{eq:fiber2} with $F \in {\cal M}_H(u_1,p\Pi_i)^s 
\cup {\cal M}_H(2u_1,pf)^s \cup
{\cal M}_H(u_2,p' f)^s$
and $\widetilde{E} \in {\cal M}_f(v-v(F))_*^{ss}$,
%
\begin{equation}
\Ext^2(F,{\cal E}_{|X \times \{ x\}})=
\Hom({\cal E}_{|X \times \{ x\}},F)^{\vee}=0
\end{equation}
for all $x \in X$
by \eqref{eq:FM1}, $\frac{1}{2}=\frac{d}{r}>\frac{q}{p}=0$ and
$\frac{1}{2}=\frac{d}{r}>\frac{q'}{p'}=\frac{1}{4}$.
Since $\Hom(F,{\cal E}_{|X \times \{ x\}})=0$
for a general $x \in X$, 
we see that 
$$
\Psi(\widetilde{E})[1], \Psi(E)[1],\Psi(F)[1] \in \Coh(X)
$$
and we have an exact sequence
$$
0 \to \Psi^1(F) \to \Psi^1(E) \to \Psi^1(\widetilde{E}) \to 0.
$$
By using Lemma \ref{lem:Psi},
we have the following description of the boundary divisors.

\begin{enumerate}
\item
For a general member $E \in \Psi^1({\cal D}_i)$ $(i=1,2)$,
$\Div(E)=\Pi_i+C$, where $C$ is flat over $\pi$.
\item
For a general member $E \in \Psi^1({\cal D}_i)$ $(i=3,4)$,
$\Div(E)=f+C$, where 
$C$ is flat over $\pi$.
\end{enumerate}
By this description of the boundary,
we have the following claim
\begin{prop}[{cf.\cite[Assumption 2.16]{Sacca}}]\label{prop:sacca}
We set
$$
|L|_{nr}:=\{D \in |L| \mid \text{ $D$ is not reduced}\}
$$
and 
${\cal M}_{H'}^{G'}(w,L)_{nr}=\varphi^{-1}(|L|_{nr})$.
Then 
$\codim_{{\cal M}_{H'}^{G'}(w,L)}({\cal M}_{H'}^{G'}(w,L)_{nr}) \geq 2$.
\end{prop}
\begin{NB}
Hence \cite[Thm. 5.1]{Sacca} implies the following.
\begin{thm}
Let $X$ be an Enriques surface such that \eqref{eq:rho=10} holds.
For a primitive Mukai vector $v=(r,L,\frac{s}{2})$ such that $\ell(v)=1$ and
$\langle v^2 \rangle \geq 4$,
 $b_2(M_H(v))=11$, where $H$ is a general polarization. 
\end{thm}
\end{NB}

We next treat case (2).
Since $\Psi(v)=2\Psi((2,e_2+(n+1)e_1,0))+(0,2e_1,1)$ by
Lemma \ref{lem:Psi},
we set $w=(0,2\xi,a)$ with $(2,a)=1$.
Let $L$ be a divisor with $[L \mod K_X]=\xi$.
We shall prove the irreducibility of ${\cal M}_{H'}^{G'}(w,2L)^{ss}$
and ${\cal M}_{H'}^{G'}(w,2L+K_X)^{ss}$.
In order to prove the irreducibility of ${\cal M}_{H'}^{G'}(w,2L)^{ss}$,
we consider the support map
$\varphi:{\cal M}_{H'}^{G'}(w,2L)^{ss} \to |2L|$.
We set
\begin{equation}
\begin{split}
N_1:= & \{ D \in |2L| \mid D=2C \},\\
N_2:= & \{ D \in |2L| \mid D=C_1+C_2, (C_1,e_1)=(C_2,e_1)=1, \, C_1 \ne C_2 \}.
\end{split}
\end{equation}

\begin{equation}
{\cal M}_i:=\{ E \in {\cal M}_{H'}^{G'}(w,2L)^{ss}_* \mid \Div(E) \in N_i \}. 
\end{equation}

\begin{lem}
$\dim {\cal M}_1 \leq \frac{7}{2}(L^2)-1$.
\end{lem}

\begin{proof}
Let $E \in {\cal M}_{H'}^{G'}(w,2L)^{ss}$ satisfy 
$\Div(E)=2C$, $C \in |L +\epsilon K_X|$ $(\epsilon=0,1)$.
If $E$ is not an ${\cal O}_C$-module, then we have an exact sequence
\begin{equation}\label{eq:2C}
0 \to E_1 \to E \to E_0 \to 0
\end{equation}  
and an injective homomorphism
$E_0(-C) \to E_1$.
Indeed, let $T$ be the 0-dimensional submodule of $E_{|C}$.
Then $E_0=E_{|C}/T$ and $E_1=\ker(E \to E_0)$.
Since $E_1$ is purely 1-dimensional,
$E_1/(E_0(-C))$ is $0$-dimensional
and $E_0(-C)$ is an ${\cal O}_C$-module,
$E_1$ is also an ${\cal O}_C$-module.
\begin{NB}
\begin{equation}
\begin{CD}
E_1 @>{C}>> E_1(C)\\
@AAA @AAA\\
E_0(-C) @>>> 0
\end{CD}
\end{equation}
\end{NB}
If $C$ is smooth, the locus of $E$ fitting into
\eqref{eq:2C} is of dimension 
$2(C^2)=4g(C)-4$ by \cite[Prop. 2.1]{Inaba}. Since $\dim |C|=(C^2)/2$,
it is of dimension $5(L^2)/2$. 
Assume that $C$ is singular.
We set $v(E_i)=v_i:=(0,C,a_i)$.
Since $v(E_0(-C))=(0,C,a_0-(C^2))$, we have
$$
v(E_1)=v_1=(0,C,a_0-(C^2)+k),\; k \geq 0.
$$
Since $2a_0-(C^2)+k=a$ is odd,
$k \geq 1$.
\begin{NB}
If $a$ is even, we have $k \geq 0$.
\end{NB}
Since $E_0$ and $E_1$ are stable sheaves on $C$,
$$
\dim \Hom(E_1,E_0(K_X)) \leq 
\chi(E_0)-\chi(E_1)+1=(C^2)-k+1 \leq (C^2).
$$
\begin{NB}
Since $K_X$ is numerically frivial,
$\chi(E_0(K_X))=\chi(E_0)$.
\end{NB}
\begin{NB}
If $a$ is even, we have 
$$
\dim \Hom(E_1,E_0(K_X)) \leq 
\chi(E_0)-\chi(E_1)+1=(C^2)-k+1 \leq (C^2)+1.
$$
\end{NB}
We set
\begin{equation}
M_i:=\{E \in {\cal M}_{H'}^{G'}(v_i,C)_*^{ss} \mid \text{ $\Div(E)$ is singular} \}. 
\end{equation} 
Since all fibers of $\varphi:{\cal M}_{H'}^{G'}(v_i,C)_*^{ss} \to |C|_*$ 
are of dimension $(C^2)/2$ by \cite{AIK},
we get
\begin{equation}
\dim M_0 \times_{|C|} M_1+\langle v_0,v_1 \rangle+(C^2) \leq
\frac{7}{2}(C^2)-1,
\end{equation}
which shows the locus of $E$ fitting into
\eqref{eq:2C} such that $E_i \in M_i$ is at most 
of dimension $7(C^2)/2-1$ by Lemma \ref{lem:A-est1}. 
\begin{NB}
$\dim M_0 \times_{|C|} M_1 \leq (C^2)/2-1+2((C^2)/2)$,
since all fibers of $M_i \to |C|$ is of dimension 
$g(C)-1=(C^2)/2$ by \cite{AIK}. 

More rough estimate:
\begin{equation}
\dim M_0+\dim M_1+\langle v_0,v_1 \rangle+(C^2) \leq
4(C^2)-2.
\end{equation}
If $a$ is even, then
$\dim M_0 \times_{|C|} M_1+\langle v_0,v_1 \rangle+(C^2) \leq
\frac{7}{2}(C^2)$.
\end{NB}

Assume that $E$ is an ${\cal O}_C$-module.
If $C$ is smooth, then
$E$ is a stable locally free sheaf of rank 2 on $C$.
Hence 
the dimension is 
$\dim |C|+4(g(C)-1)=5(C^2)/2$.
We assume that $C$ is singular.
We set $a=2k+1$. 
Let ${\cal O}_C(1)$ be a line bundle of degree 1 on $C$.
Since $\chi(E(-k))=1$,
we have a homomorphism ${\cal O}_C(k) \to E$.
Hence we get an exact sequence
\begin{equation}\label{eq:CC}
0 \to E_1 \to E \to E_0 \to 0
\end{equation}
such that $E_0$ and $E_1$ are torsion free of rank 1 and
$E_1$ contains ${\cal O}_C(k)$.
We set $v(E_i)=v_i=(0,C,a_i)$.
Then 
\begin{equation}
\begin{split}
a_1=& \chi(E_1)=\chi({\cal O}_C(k))+l\; (l \geq 0),\\
a_0=& a-a_1=2k+1-\chi({\cal O}_C(k))-l.
\end{split}
\end{equation}
Since $\chi({\cal O}_C(k))=k-\frac{(C^2)}{2}$,
we have 
$$
\chi(E_0)-\chi(E_1)=(2k+1)-2\chi({\cal O}_C(k))-2l=
1+(C^2)-2l.
$$
Hence $\dim \Hom(E_1,E_0(K_X)) \leq
(C^2)+2-2l$.
\begin{NB}If $a$ is even, we set $a=2k+2$.
Then $$
\chi(E_1)-\chi(E_0)=(2k+1)-2\chi({\cal O}_C(k))-2l=
1+(C^2)-2l+1.
$$
\end{NB}
If $l \geq 1$, then
\begin{equation}
\dim M_0 \times_{|C|} M_1+\langle v_0,v_1 \rangle+(C^2) \leq
\frac{7}{2}(C^2)-1.
\end{equation}
If $l=0$, then $E_1={\cal O}_C(k)$.
Hence 
\begin{equation}
-1+\dim M_1+\langle v_0,v_1 \rangle+(C^2)+2 \leq
-1+(C^2)-1+2(C^2)+2=3(C^2).
\end{equation}
Since $(C^2)/2 \geq 1$,
$7(C^2)/2-1-3(C^2) \geq 0$.
Hence the locus of $E$ fitting into \eqref{eq:CC} is at most of 
dimension $7(C^2)/2-1$ by Lemma \ref{lem:A-est1}.
Therefore we get the claim.
\end{proof}

\begin{lem}\label{lem:reducible}
$\dim {\cal M}_2 \leq 4(L^2)-2$.
\end{lem}

\begin{proof}
For $E \in  {\cal M}_2$, we have an exact sequence
\begin{equation}
0 \to E_1 \to E \to E_2 \to 0
\end{equation}
where $\Div E=C_1+C_2$, $\Div(E_1)=C_1$ and $\Div(E_2)=C_2$.
By the flatness of $\Div(E)$,
$C_1$ and $C_2$ are integral curves and $C_1 \ne C_2$. 
By the stability of $E$,
we have
\begin{equation}\label{eq:chi}
\frac{a}{(2L,H)} \leq \frac{\chi(E_2)}{(C_2,H)},\;
\frac{\chi(E_2)-(C_1,C_2)}{(C_2,H)} \leq \frac{a}{(2L,H)}.
\end{equation}
\begin{NB}
$E_2:=E_{|C_2}/T$, where $T$ is the torsion submodule.
$E(-C_1) \to E$ induces a homomorphism
$E_{|C_2}(-C_1) \to E$.
Since $E$ is purely 1-dimensional,
it induces an injection $E_2(-C_1) \to E$.
\end{NB}
We set $v_i:=v(E_i)=(0,C_i,a_i)$.
By \eqref{eq:chi},
 the choice of $v_i$ is finite.
Since $(C_i,e_1)=1$,
$v_i$ are primitive with
$\ell(v_i)=1$. Hence
$\dim {\cal M}_{H'}^{G'}(v_i)^{ss}=\langle v_i^2 \rangle=(C_i^2)$.
\begin{NB}
Even if $v_i$ is isotropic,
$\dim {\cal M}_{H'}^{G'}(v_i)=\langle v_i^2 \rangle$.
\end{NB}
Let $J(v_1,v_2)$ be the substack of ${\cal M}_2$
such that $E_i \in {\cal M}_{H'}^{G'}(v_i,C_i)^{ss}$.
Since $\Hom(E_1,E_2(K_X))=0$,
\begin{equation}
\begin{split}
\dim J(v_1,v_2) =& \dim {\cal M}_{H'}^{G'}(v_1,C_1)^{ss}
+\dim {\cal M}_{H'}^{G'}(v_2,C_2)^{ss}+(C_1,C_2)\\
=&
(C_1^2)+(C_2^2)+(C_1,C_2)=4(L^2)-(C_1,C_2).
\end{split}
\end{equation}
We note that $(C_1^2),(C_2^2) \geq 0$.
If $(C_1^2),(C_2^2)>0$, then
$(C_1,C_2)^2 \geq (C_1^2)(C_2^2) \geq 4$.
Hence $(C_1,C_2) \geq 2$.
If one of $(C_i^2)=0$, then
$(C_1,C_2)=(2L,C_i)-(C_i^2)=2(L,C_i) \geq 2$. 
Therefore 
$\dim J(v_1,v_2) \leq 4(L^2)-2$.
\end{proof}

We set 
${\cal M}_{H'}^{G'}(w,2L)^{ss}_0:=
{\cal M}_{H'}^{G'}(w,2L)^{ss}_* \setminus ({\cal M}_1 \cup {\cal M}_2)$.
\begin{prop}\label{prop:case(2-1)}
${\cal M}_{H'}^{G'}(w,2L)^{ss}_0$ is an open and dense substack of
${\cal M}_{H'}^{G'}(w,2L)^{ss}$.
In particular, it is irreducible.
\end{prop}

For any $C_1+C_2 \in |2L+K_X|$,
$C_1 \ne C_2$.
We have a similar estimate as in Lemma \ref{lem:reducible}.
Hence we get
\begin{prop}\label{prop:case(2-2)}
${\cal M}_{H'}^{G'}(w,2L+K_X)^{ss}$ is irreducible.
\end{prop}
By Propositions \ref{prop:case(1)}, \ref{prop:case(2-1)} 
and \ref{prop:case(2-2)},
Theorem \ref{thm:intro:irred} holds, if $X$ is unnodal.  

\begin{rem}\label{rem:non-prim1}
For $w=(0,2\xi,2b)$,
we define ${\cal M}_i (\subset {\cal M}_{H'}^{G'}(w,2L)^{ss}_*)$ 
in a similar way.
Then we see that
$\dim {\cal M}_1 \leq \frac{7}{2}(L^2)$ and
$\dim {\cal M}_2 \leq 4(L^2)-2$.
Hence ${\cal M}_{H'}^{G'}(w,2L)^{ss}$ and 
${\cal M}_{H'}^{G'}(w,2L+K_X)^{ss}$
are irreducible.
In particular, ${\cal M}_{H}(2v,L)^{ss}$ 
is irreducible, where $v$ is primitive, $\rk v$ is even and
$\ell(v)=1$. 
\end{rem}

\begin{rem}\label{rem:non-prim2}
For a divisor $D':=2^m D+K_X$ such that $(D \in \NS(X))$ is primitive,
$D'$ is not linearly equivalent to
$pC$, $p \geq 2$.
Hence we also see that 
${\cal M}_{H}(2^m v,L)^{ss}$ is irreducible,
if $v$ is primitive and
$L \ne 2L'$, $L' \in \NS(X)$.
\begin{NB}
$E=2F$ iff
$\Psi(E)=2\Psi(F)$, and
$E=2F$ iff $L= 2L'$ for $E \in {\cal M}_H(2^mv,L)^{ss}$. 
\end{NB}  
\end{rem}

\subsection{General cases}

We shall treat a general case by the arguments in \cite[sect. 3]{Enriques}.
Let $({\cal X}, {\cal H}) \to S$ be a general
deformation of $(X,H)$ such that
a general member is not nodal and
$({\cal X}_0,{\cal H}_0)=(X,H)$ $(0 \in S)$.
Let ${\cal L}$ be a family of divisors such that
${\cal L}_0=L \in \NS(X)$.
Then we have a family of moduli spaces of semi-stable
sheaves
$\phi:{M}_{({\cal X},{\cal H})}(v,{\cal L}) \to S$.
Since $\Pic({\cal X}_s)=H^2({\cal X}_s,{\Bbb Z})$ is locally constant,
we may assume that ${\cal H}_s$ is general with respect to $v$ for all 
$s \in S$. 
Let $M_0$ be the open subscheme of 
${M}_{({\cal X},{\cal H})}(v)$ such that 
$E \in \Coh({\cal X}_s)$ belongs to $M_0$ 
if and only if $\Hom(E,E(K_{{\cal X}_s}))=0$. 
Then $M_0$ is smooth over $S$.
By Lemma \ref{lem:s}, $M_0$ is
a dense subscheme of  
${M}_{({\cal X},{\cal H})}(v,{\cal L})$.
Since $(M_0)_s$ is irreducible for any unnodal
surface ${\cal X}_s$,
$M_0$ is irreducible, which implies
${M}_{({\cal X},{\cal H})}(v,{\cal L})$ is also irreducible.
By the Zariski connectedness theorem,
all fibers are connected.
In particular, $M_H(v,L)$ is connected.
If $\langle v^2 \rangle \geq 4$, then
$M_H(v,L)$ is irreducible by
Lemma \ref{lem:s} and Lemma \ref{lem:pss} (1).
Therefore we get the following. 
\begin{thm}\label{thm:irred}
Let $v=(r,\xi,a)$ 
be a primitive Mukai vector such that $r$ is even.
Then ${\cal M}_H(v,L)^{ss}$ is connected for a general $H$.
Moreover  if $\langle v^2 \rangle \geq 4$, then
${\cal M}_H(v,L)^{ss}$ is irreducible for a general $H$.
\end{thm}

\begin{rem}\label{rem:irred}
Let $v$ be a primitive Mukai vector.
By \cite[Rem. 4.1]{Y:twist1},
the proof of \cite[Thm. 2.6]{Enriques} and
\cite[Rem. 2.19]{Enriques}, we have
$$
e({\cal M}_H(mv)^{ss})=e({\cal M}_H(m w)^{ss}),
$$
where $w=(1,0,-\frac{s}{2})$ if $\rk v$ is odd, 
and $w=(2,\xi_i,-\frac{s}{2})$
($1 \leq i \leq 2^{10}$)
with
$$
\{ \xi_i \mod 2 | 1 \leq i \leq 2^{10} 
\}=\NSf(X) \otimes {\Bbb Z}/2{\Bbb Z}=
({\Bbb Z}/2{\Bbb Z})^{\oplus 10}
$$
 if $\rk v$ is even.
By the works of Gieseker-Li \cite{G-L} or O'Grady \cite{O},
moduli stacks are asymptotically irreducible:

If $\rk w=1$, then 
there is $N_1(m)$ such that 
${\cal M}_H(mw,L)^{ss}$ is irreducible
for $\langle w^2 \rangle \geq  N_1(m)$.
Assume taht $\rk w=2$.
For each $(m \rk w,m \xi_i)$, there is $N(m,\xi_i)$ such that 
${\cal M}_H(mw,L)^{ss}$ is irreducible if
$\langle w^2 \rangle \geq  N(m,\xi_i)$.
We set $N_2(m):=\max_i N(m,\xi_i)$.
Then 
${\cal M}_H(mv,L)^{ss}$ is irreducible if
$\langle v^2 \rangle \geq N(m):=\max\{N_1(m),N_2(m)\}$.
\end{rem}

\section{Appendix}\label{sect:appendix}

\subsection{}
Let $\pi:X \to C$ be an elliptic surface.
Let ${\bf e} \in K(X)_{\topo}$ be the class of a coherent sheaf $E$
with $\rk E=r$ and $(c_1(E),f)=d$.
Assume that $\gcd(r,d)=1$.
Then ${\cal M}_{H_f}({\bf e})^{ss}={\cal M}_{H_f}({\bf e})^{s}$
is smooth of dimension 
$-\chi({\bf e},{\bf e})+p_g$, 
where $p_g:=\dim H^2(X,{\cal O}_X)$ is the geometric genus of $X$.
In this case, Bridgeland showed that a 
suitable relative Fourier-Mukai transform
induces a birational map between $M_{H_f}({\bf e})$ and
the moduli of stable sheaves of rank 1.
In this section, we shall slightly refine the correspondence. 
We assume that every fiber is irreducible and 
there is no multiple fiber.  
Then we have a refinement of Proposition \ref{prop:codim-general}.
Let 
$$
0 \subset F_1 \subset F_2 \subset \cdots \subset F_s=F
$$
be the filtration in \eqref{eq:HN}.
We set 
$$
(c_1(F_i/F_{i-1}),\chi(F_i/F_{i-1})):=
l_i(r_i f,d_i),
$$
 where $l_i,r_i,d_i \in {\Bbb Z}$,
$\gcd(r_i,d_i)=1$ and $l_i>0$.
We set $\mu_{\min}(E):=d_s /r_s$.
We define ${\cal F}(\widetilde{\bf e},{\bf f}_1,...,{\bf f}_s)$
as in Proposition
\ref{prop:codim-general}.
Then we have 
\begin{prop}\label{prop:codim}
$\codim {\cal F}(\widetilde{\bf e},{\bf f}_1,...,{\bf f}_s)
=\sum_{i} l_i((r_i d-rd_i)-1)$.
\end{prop}

\begin{proof}
\begin{NB}
We first note that
$E_i \otimes \omega_X \cong E_i$
for $E_i \in {\cal M}_H({\bf f}_i)^{ss}$.
By \eqref{eq:F_i}, \eqref{eq:{E}} and the Serre duality,
we have
\begin{equation}
\begin{split}
\Ext^2(F_j/F_{j-1},F_i/F_{i-1})=0,\;i<j\\
\Ext^2(F_i/F_{i-1},\widetilde{E})=0,\;1 \leq i \leq s.
\end{split}
\end{equation}
Then the proof of \cite[Lem. 5.3]{K-Y} implies that
\begin{equation}
\begin{split}
\dim {\cal F}(\widetilde{\bf e},{\bf f}_1,...,{\bf f}_s)
=& -\sum_i \chi({\bf f}_i,\widetilde{\bf e})-
\sum_{i<j}\chi({\bf f}_j,{\bf f}_i)+
\dim {\cal M}_{H_f}(\widetilde{\bf e})^{ss}+
\sum_i {\cal M}_H({\bf f}_i)^{ss}\\
=& -\sum_i \chi({\bf f}_i,\widetilde{\bf e})+
\dim {\cal M}_{H_f}(\widetilde{\bf e})^{ss}+
\sum_i {\cal M}_H({\bf f}_i)^{ss}\\
=& \sum_i l_i(r_i d-rd_i)+\dim {\cal M}_{H_f}(\widetilde{\bf e})^{ss}
+\sum_i l_i.
\end{split}
\end{equation}
\end{NB}
By Proposition \ref{prop:codim-general} and
$\dim {\cal M}_H({\bf f}_i)^{ss}=l_i$,
\begin{equation}
\dim {\cal F}(\widetilde{\bf e},{\bf f}_1,...,{\bf f}_s)
= \sum_i l_i(r_i d-rd_i)+\dim {\cal M}_{H_f}(\widetilde{\bf e})^{ss}
+\sum_i l_i.
\end{equation}
Since $\chi(\widetilde{\bf e},{\bf f}_i)=
\chi({\bf f}_i,\widetilde{\bf e})$,
we get
\begin{equation}
\begin{split}
\dim {\cal M}_{H_f}({\bf e})^{ss}=&
-\chi({\bf e},{\bf e})+p_g\\
=&-\chi(\widetilde{\bf e},\widetilde{\bf e})-
\sum_i 2\chi(\widetilde{\bf e},{\bf f}_i)+p_g\\
=& \dim {\cal M}_{H_f}(\widetilde{\bf e})^{ss}+
2\sum_i l_i(r_i d-rd_i).
\end{split}
\end{equation}
Hence the claim holds.
\end{proof}

Let $Y:=M_H(0,r' f, d')$ be a fine moduli space of stable sheaves
on $X$ and
${\bf P}$ a universal family on
$X \times Y$.
$p_X:X \times Y \to X$ and $p_Y:X \times Y \to Y$
denote the projections.  
We consider a contravariant functor
\begin{equation}\label{eq:(r',d')}
\begin{matrix}
\Phi_{X \to Y}^{{\bf P}} \circ D_X:& {\bf D}(X) &
\to & {\bf D}(Y)\\
& E & \mapsto &
{\bf R}\Hom_{p_Y}(p_X^*(E),{\bf P}).
 \end{matrix}
\end{equation}
By the Grothendieck-Serre duality,
$D_Y \circ \Phi_{X \to Y}^{{\bf P}} \cong
\Phi_{X \to Y}^{{\bf P}^{\vee}[2] \otimes p_X^*(\omega_X)}
\circ D_X$.
Hence $\Phi_{Y \to X}^{{\bf P}} \circ D_Y$ is the
inverse of $\Phi_{X \to Y}^{{\bf P}} \circ D_X$.
\begin{NB}
We note that ${\bf P} \otimes p_X^*(\omega_X)=
{\bf P} \otimes p_Y^*(\omega_Y)$.
Since $\Phi_{Y \to X}^{{\bf P}^{\vee}[2] \otimes p_Y^*(\omega_Y)}
\Phi_{X \to Y}^{{\bf P}}=\id_X$ and
$\Phi_{X \to Y}^{{\bf P}}
\Phi_{Y \to X}^{{\bf P}^{\vee}[2] \otimes p_X^*(\omega_X)}
=\id_Y$, we get 
$\Phi_{Y \to X}^{{\bf P}} \circ D_Y 
\Phi_{X \to Y}^{{\bf P}} \circ D_X=\id_X$ and
$\Phi_{X \to Y}^{{\bf P}} \circ D_X
\Phi_{Y \to X}^{{\bf P}} \circ D_Y=\id_X$.
\end{NB}
Assume that
$r'd-rd'>0$.
Then

\begin{lem}\label{lem:isom}
For $E \in {\cal M}_{H_f}({\bf e})^{ss}$,
$\Phi_{X \to Y}^{{\bf P}}(E^{\vee})[1] \in \Coh(Y)$.
Moreover if $\mu_{\min}(E) \geq d'/r'$, then
$\Phi_{X \to Y}^{{\bf P}}(E^{\vee})[1]$ is torsion free.
In particular, $\Phi_{X \to Y}^{{\bf P}}(E^{\vee})[1]$ is
stable.
\end{lem}

\begin{proof}
We note that
$$
{\bf P}_{|X \times \{y \}} \otimes K_X \cong {\bf P}_{|X \times \{y \}}
$$
for all $y \in Y$ by the general theory
of Fourier-Mukai transforms \cite{Br:1}.
By the Serre duality and the torsion freeness of $E$,
$$
\Ext^2(E,{\bf P}_{|X \times \{y \}})=
\Hom({\bf P}_{|X \times \{y \}},E)^{\vee}=0.
$$
Hence $H^2(\Phi_{X \to Y}^{{\bf P}}(E^{\vee}))=0$.
We note that
$\Hom(E,{\bf P}_{|X \times \{y \}})=0$, if
$E_{|\pi^{-1}(\pi(y))}$ is semi-stable.
Since $E_{|f}$ is semi-stable for a general fiber
of $\pi$,
$\Hom(E,{\bf P}_{|X \times \{y \}})=0$ for a general $y \in Y$.
Since $H^0(\Phi_{X \to Y}^{{\bf P}}(E^{\vee}))$ is torsion free,
$H^0(\Phi_{X \to Y}^{{\bf P}}(E^{\vee}))=0$.
Therefore 
$\Phi_{X \to Y}^{{\bf P}}(E^{\vee})[1] \in \Coh(Y)$.

Assume that $\mu_{\min}(E) \geq d'/r'$.
If $\Hom(E,{\bf P}_{|X \times \{y \}}) \ne 0$,
then $F_s/F_{s-1}$ in \eqref{eq:HN} is a semi-stable sheaf
with $\mu(F_s/F_{s-1})= d'/r'$ and
we have a surjective homomorphism
$F_s/F_{s-1} \to {\bf P}_{|X \times \{y \}}$.
Assume that $F_s/F_{s-1}$ is $S$-equivalent to $\oplus_{i=1}^k E_i$,
where $E_i$ are stable 1-dimensional sheaves with
$\mu(E_i)=d'/r'$. 
Then ${\bf P}_{|X \times \{y \}} \in \{E_1,...,E_k \}$.
Therefore $H^1(\Phi_{X \to Y}^{{\bf P}}(E^{\vee}))$ is
torsion free.
\end{proof}

\begin{NB}
\begin{cor}
We set
$$
M_{H_f}({\bf e})_0:=\{E \in M_{H_f}({\bf e})| \mu_{\min}(E) \geq d'/r' \}.
$$
Then we have a morphism
\begin{equation}
\begin{matrix}
M_{H_f}({\bf e})_0 & \to & M_{H'_f}({\bf e'})\\
E & \mapsto & \Phi_{X \to Y}^{{\bf P}}(E^{\vee})[1],
\end{matrix}
\end{equation}
where ${\bf e}' \in K(Y)_{\topo}$ is the class
of $\Phi_{X \to Y}^{{\bf P}}(E^{\vee})[1]$.
\end{cor}

By Proposition \ref{prop:codim},
we have the following. 
\begin{lem}
$$
\codim_{M_{H_f}({\bf e})}(M_{H_f}({\bf e}) \setminus M_{H_f}({\bf e})_0)
\geq r_{\min} d-rd_{\min}-1,
$$
where $r_{\min}>0$ and $d_{\min}$ are relatively prime
integers with $d_{\min}/r_{\min}=\mu_{\min}(E)$.
\end{lem}
\end{NB}

\begin{prop}\label{prop:Hilb}
Let ${\bf e}' \in K(Y)$ be the class 
of an ideal sheaf $I_Z \in \Hilb_Y^b$.
Then
there is a (contravariant) Fourier-Mukai transform
${\bf D}(X) \to {\bf D}(Y)$ which induces an isomorphism
$$
{\cal M}_{H_f}({\bf e})^{ss} \setminus {\cal Z} \to 
{\cal M}_{H'_f}({\bf e}')^{ss} \setminus {\cal Z}',
$$
where $2b=-\chi({\bf e},{\bf e})+\chi({\cal O}_X)$,
${\cal Z} \subset {\cal M}_{H_f}({\bf e})^{ss}$
and 
${\cal Z}' \subset {\cal M}_{H'_f}({\bf e}')^{ss}$ are
closed substacks with
$$
\dim {\cal Z}, \dim {\cal Z}' \leq \dim {\cal M}_{H'_f}({\bf e})^{ss}-2.
$$
\end{prop}

\begin{proof}
Let $(p,q)$ be a pair of integers such that 
$dp-rq=1$ and $0<p<r$.
We first assume that $p \leq r/2$.
If two integers $x,y$ safisfy $dx-ry=m$ $(m=1,2)$, then
$(x,y)=m(p,q)+n(r,d)$ where $n \in {\Bbb Z}$.
If $0<x \leq r$, then we get $n=0$, i.e., $(x,y)=m(p,q)$.
If a pair $(x,y)$ of integers safisfy $x>0$ and $y/x<d/r$, then
we have $y/x \leq q/p$ or $x \geq r+p>r$ by
\cite[Lem. 5.1]{Y:K3}.

For the filtration \eqref{eq:HN},
we have $0<r_i \leq r$. Hence we have
\begin{equation}
\frac{q}{p} \geq \frac{d_1}{r_1}>\frac{d_2}{r_2}> \cdots >\frac{d_s}{r_s}.
\end{equation}
Assume that 
\begin{equation}
\codim {\cal F}(\widetilde{\bf e},{\bf f}_1,...,{\bf f}_s)
=\sum_{i} l_i((r_i d-rd_i)-1) \leq 1.
\end{equation}
Then $r_i d-r d_i=1,2$ and $0<r_i \leq r$ imply that
$(r_i,d_i)=(p,q),(2p,2q)$.
We set
$$
{\cal Z}:=\{E \in {\cal M}_{H_f}({\bf e})^{ss} \mid
\mu_{\min}(E)<q/p\}.
$$
Then $\dim {\cal Z} \leq \dim {\cal M}_{H_f}({\bf e})^{ss}-2$
by Proposition \ref{prop:codim}.
For $(r',d')=(p,q)$, we consider the Fourier-Mukai
transform \eqref{eq:(r',d')}.
Since $\rk \Phi_{X \to Y}^{{\bf P}[1]}({\bf e}^{\vee})=1$,
we may assume that $\Phi_{X \to Y}^{{\bf P}[1]}({\bf e}^{\vee})={\bf e}'$.
We note that $\tau({\bf P}_{|\{x \} \times Y})=(0,r'f,-r)$.
We set
$$
{\cal Z}':=\{E \in {\cal M}_{H_f}({\bf e}')^{ss} \mid
\mu_{\min}(E)<-r \}.
$$
Since $r \geq 2$, we have
$\dim {\cal Z}' \leq \dim {\cal M}_{H_f}({\bf e})^{ss}-2$ by
Proposition \ref{prop:codim}.
Hence $\Phi_{X \to Y}^{{\bf P}[1]} \circ D_X$ induces an isomorphism
$$
{\cal M}_{H_f}({\bf e})^{ss} \setminus {\cal Z} \to
{\cal M}_{H'_f}({\bf e}')^{ss} \setminus {\cal Z}'
$$
by Lemma \ref{lem:isom}.
We next assume that $p>r/2$.
In this case, we have $r \geq 3$.
Then
the closed substack $W$ of ${\cal M}_{H_f}({\bf e})^{ss}$
consisting of non-locally free sheaves is of codimension $r-1 \geq 2$.
For the topological invariant ${\bf e}^{\vee}$, 
$(x,y)=(r-p,q-d)$ safisfies
$0<r-p<r/2$ and $(-d)x-ry=1$.
Applying the first part, we have a similar isomorphism.
Thus $\Phi_{X \to Y}^{{\bf P}^{\vee}[1]}$ induces an isomorphism
$$
{\cal M}_{H_f}({\bf e})^{ss} \setminus {\cal Z} \to 
{\cal M}_{H'_f}({\bf e}')^{ss} \setminus {\cal Z}',
$$
where $(r',d')=(r-p,d-q)$,
${\cal Z}$ consists of $E$ which is non locally free
or $\mu_{\min}(E) <(d-q)/(r-p)$ and
${\cal Z}'$ consists of $F$ with $\mu_{\min}(F) \leq -r$. 
\end{proof}

For ${\bf e} \in K(X)_{\topo}$,
we set 
$$
K(X)_{\bf e}:=\{\alpha \in K(X) \mid \chi(\alpha,{\bf e})=0 \}.
$$
We have a homomorphism
\begin{equation}
\begin{matrix}
\theta_{\bf e}:& K(X)_{\bf e} & \to & \Pic(M_{H_f}({\bf e}))\\ 
& \alpha & \mapsto & \det p_{!}({\cal E} \otimes p_X^*(\alpha^{\vee})),
\end{matrix}
\end{equation}
where ${\cal E}$ is a universal family. We note that
$\theta_{\bf e}$ can be defined even if there is no universal family
by using a family on a quot-scheme.
For the Fouruer-Mukai transform $\Phi$ in Proposition
\ref{prop:Hilb}, we have a commutative diagram
\begin{equation}\label{eq:comm}
\begin{CD}
K(X)_{\bf e} @>{\Phi}>> K(Y)_{{\bf e}'}\\
@V{\theta_{\bf e}}VV @VV{\theta_{{\bf e}'}}V\\
\Pic(M_{H_f}({\bf e})) @= \Pic(M_{H'_f}({\bf e}'))
\end{CD}.
\end{equation}

\begin{cor}
Assume that $\dim M_{H_f}({\bf e})\geq 4+q(X)$ and $k={\Bbb C}$.
Then we have an exact sequence
$$
0 \longrightarrow \ker \tau \longrightarrow
 K(X)_{\bf e} \overset{\theta_{\bf e}}{\longrightarrow}
 \Pic(M_{H_f}({\bf e}))/\Pic(\Alb(M_{H_f}({\bf e}))) \longrightarrow 0.
$$
\end{cor}

\begin{proof}
By Proposition \ref{prop:Hilb} and \eqref{eq:comm},
it is sufficient to prove the claim for ${\bf e}'$.
In this case, 
we note that
$\pi_1(\Hilb_Y^b) \cong \pi_1(Y)$ (\cite[sect. 1]{OS}),
$H^1(\Hilb_Y^b,{\Bbb Z})=H^1(Y,{\Bbb Z})$ and
$$
H^2(\Hilb_Y^b,{\Bbb Z})=H^2(Y,{\Bbb Z}) \oplus \wedge^2 H^1(Y,{\Bbb Z})
\oplus {\Bbb Z}\delta,
$$
where $2\delta$ is the exceptional divisor of the Hilbert-Chow map.
Then it is easy to see that
the claim holds (cf. \cite[sect. 3.2]{univ}).
\end{proof}

\begin{NB}
Assume that $0<r' \leq r$.
If $r' d-rd' \geq 2$, then
there is a pair $(r'',d'')$ of integers
such that $0<r'' \leq r$ and $r'' d-d'' r=1$.
Then $d/r>d''/r'' \geq d'/r'$ and 
$d/r<(d-d'')/(r-r'') \leq (d-d')/(r-r')$.
Replacing $(r',d')$ by $(r'',d'')$, 
we may asume that $r' d-rd'=1$.
Since $0 \leq r-r_{\min}<r$,
 $r_{\min} d-rd_{\min}=1,2$ if and only if
$(r_{\min},d_{\min})=(2r''-r,2d''-d),
(r'',d''), 2(r'',d'')$.
Obviously $(r_{\min},d_{\min})=(r'',d'')$ does
not hold. 
If $r \leq 2r''$, then
$(r_{\min},d_{\min})=(2r''-r,2d''-d)$
and $d'/r'>d_{\min}/r_{\min}$.
In this case, $(d-d_{\min})/(r-r_{\min})<
(d-d')/(r-r')$.

If $r \geq 2r''$, then
$(r_{\min},d_{\min})=2(r'',d'')$
and $d_{\min}/r_{\min}=d''/r'' \geq d'/r'$. 

In particular, if $r \geq 2r''$, then
$r_{\min} d-rd_{\min} \geq 3$.

If $r''>2r$, then $r \geq r''>2$.
Hence the locus of non-locally sheaves of
$M_{H_f}(r,\xi,a)$ is at lease of codimension 2.
We set
$Y:=M_H(0,(r-r')f,d-d')$ and consider 
$\Phi_{X \to Y}^{{\bf P}^{\vee}}$.
Then we have a birational map
$M_{H_f}(r,\xi,a) \cdots \to M_{H'_f}(1,\eta,b) \cong
\Hilb_Y^n$ which is
isomorphic up to codimension 2,
where $(\eta,f)=r$.

\begin{rem}
For a non-locally free sheaf $E$,
$H^2(\Phi_{X \to Y}^{{\bf P}^{\vee}}(E))$ has a
torsion. 
\end{rem}

If $r'' \leq 2r$, then we set
$Y:=M_H(0,r'f,d')$ and consider 
$\Phi_{X \to Y}^{{\bf P}} \circ D_X$.
Then we have a birational map
$M_{H_f}(r,\xi,a) \cdots \to M_{H'_f}(1,\eta,b) \cong
\Hilb_Y^n$ which is
isomorphic up to codimension 2, where
$(\eta,f)=r$.
\end{NB}

\begin{rem}\label{rem:unstable}
If $r'>r$ and $dr'-rd'=1$, then
$\Phi_{X \to Y}^{{\bf P}[1]}(E^{\vee})$ is not torsion free
for any $E \in {\cal M}_{H_f}({\bf e})^{ss}$.
\end{rem}

Assume that there is a multiple fiber $mf_0$.
Let $F$ be a semi-stable sheaf on $mf_0$ and
set $\tau(F)=l_i(0,r_i f_0,d_i)$.
Then $\dim {\cal M}_H(0,l_i r_i f_0,l_i d_i)^{ss} \leq [l_i m_0/m]$
by Remark \ref{rem:isotropic},
where $m_0=\gcd(r_i,m)$.
Assume that $(p,q)$ satisfies $mp \leq r$.
If $r_i d/m-r d_i=1$ $(0<r_i \leq r$), then $(r_i,d_i)=(mp,q)$.
If $r_i d/m-r d_i=2$ $(0<r_i \leq r$) and $m \mid r_i$, then we also 
have $(r_i,d_i)=(2mp,2q)$ by $0<2p,r_i/m \leq r$.
If $r_i d/m-r d_i \geq 3$, then
$$
l_i(r_i d/m-r d_i)-l_i \geq 2l_i \geq 2.
$$
If $r_i d/m-r d_i=2$ and $m_0<m$, then we have 
$$
l_i(r_i d/m-r d_i)-l_i m_0/m \geq 3l_i/2 \geq 2.
$$
Therefore there is a closed substack
${\cal Z}$ of ${\cal M}_{H_f}({\bf e})^{ss}$
such that 
$\Phi_{X \to Y}^{{\bf P}[1]}(E)$ is torsion free
for $E \in  {\cal M}_{H_f}({\bf e})^{ss} \setminus {\cal Z}$ 
and $\dim {\cal Z} \leq \dim {\cal M}_{H_f}({\bf e})^{ss}-2$.
Then we also see that Proposition \ref{prop:Hilb}
holds if $mp \leq r$.
In particular, for an unnodal Enriques surface,
Proposition \ref{prop:Hilb} holds.

Let $X$ be an unnodal Enriques surface.
As we remarked in Remark \ref{rem:unstable},
a relative Fourier-Mukai transform does not preserve stability
for any member of ${\cal M}_{H_f}(v)^{ss}$ in general.
However if $v_0$ is a primitive Mukai vector with even $\rk v_0$,
all relative Fourier-Mukai transforms preserve stability
for a general member of ${\cal M}_{H_f}(mv_0)^{ss}$ by Proposition \ref{prop:flat}.
   
In particular, if $r$ is even and $(r,(c_1,e_1))=1$, then
$M_{H_f}(r,c_1,\frac{s}{2})$ is birationally equivalent to
$M_{H_f}(2,\zeta,\frac{s'}{2})$, where $\zeta \in \Pic(X)$ and 
$(\zeta^2)-2s'=(c_1^2)-2r$.

\begin{rem}
If $r=4$, then
Nuer \cite[sect. 6]{N} constructed birational maps
of the moduli spaces by using $(-1)$-reflections.
\end{rem}

\begin{NB}
\begin{proof}
We set $d:=(c_1,e_1)$.
We note that $u:=(0,re_1,d)$ is a primitive and isotropic Mukai vector
with $\ell(u)=2$. Hence
$M_H(0,re_1,d) \cong X$ and a universal family ${\cal E}$ 
defines a Fourier
Mukai transform $\Psi_{(r,d)}$.
So we have an isomorphism
$M_{H_f}(r,c_1,\frac{s}{2}) \cong M_{H_f}^{\alpha}(0,L,b)$,
where $(L,e_1)=1$ and
$\alpha$ depends on the universal family.
\begin{NB2}
For a point $y$ in a smooth fiber $\pi^{-1}(x)$,
$(\rk E_{|\pi^{-1}(x)},\deg E_{|\pi^{-1}(x)})=(r,2d)
=2(\rk {\cal E}_{|\{y\} \times X}, \deg {\cal E}_{|\{y\} \times X})$.
\end{NB2}
On the other hand,
$M_{H_f}^{\alpha}(0,L,b)_*$ is independent of the choice of
$\alpha$.
Hence by using $\Psi^{-1}_{(2,1)}$,
we get a birational morphism
$M_{H_f}^{\alpha}(0,L,b)_* \to M_{H_f}(2,\zeta,\frac{s'}{2})$.
Thus $\Psi_{(2,1)}^{-1} \circ \Psi_{(r,d)}$ 
gives a desired birational map.
\end{proof}
\end{NB}

\begin{NB}
If $\ell(r,c_1,\frac{s}{2})=1$, then
th canonical bundle is trivial.
\end{NB}

\begin{NB}
Let $G$ be an (contravariant) autoequivalence group
of ${\bf D}(X)$ generated by $\Psi_{(r,d)}$,
where $r$ is even and $\gcd(r,d)=1$.
Then 
$\Psi \in G$ induces a birational map of
$M_{H_f}(v) \cdots \to M_{H_f}(\Psi(v))$,
where $(\rk v,(c_1(v),e_1))=l(r',d')$ with 
$r'$ is even and $\gcd(r',d')=1$.
\end{NB}

\end{document}